\documentclass[10pt]{amsart}
\usepackage[utf8]{inputenc}
\usepackage[T1]{fontenc}
\usepackage[english]{babel}

\title[Entropy and drift for Gibbs measures]{Entropy and drift for Gibbs measures on geometrically finite manifolds}

\author{Ilya Gekhtman and Giulio Tiozzo}

\address{University of Toronto \\
40 St George St \\ 
Toronto ON \\ 
Canada}

\email{ilyagekh@gmail.com}

\address{University of Toronto \\
40 St George St \\ 
Toronto ON \\ 
Canada}
\email{tiozzo@math.utoronto.ca}

\date{}
\usepackage{silence}
\usepackage{comment}
\usepackage{amsmath}
\usepackage{amssymb}
\usepackage{dsfont}
\usepackage{amsfonts}
\usepackage{amsthm}
\usepackage{tikz}
\usepackage{graphicx}
\usepackage{fancyhdr}
\usepackage{caption}
\usepackage{enumerate}
\usepackage{mathtools}
\usepackage{mathabx}
\usepackage[colorlinks=true,pdfstartview=FitH,bookmarks=false]{hyperref}
\hypersetup{urlcolor=black,linkcolor=black,citecolor=black,colorlinks=true}

\theoremstyle{definition}

\theoremstyle{plain}
\newtheorem{definition}{Definition}[section]
\newtheorem{proposition}[definition]{Proposition}
\newtheorem{corollary}[definition]{Corollary}
\newtheorem{theorem}[definition]{Theorem}
\newtheorem{lemma}[definition]{Lemma}
\newtheorem*{thm*}{Theorem}
\newtheorem*{prop*}{Proposition}
\newtheorem*{lem*}{Lemma}

\theoremstyle{remark}
\newtheorem{remark}{Remark}[section]

\newtheorem*{rem*}{Remark}

\DeclareMathOperator{\Diag}{Diag}

\begin{document}

\begin{abstract}
We prove a generalization of the fundamental inequality of Guivarc'h relating entropy, drift and critical exponent
to Gibbs measures on geometrically finite quotients of $CAT(-1)$ metric spaces.
For random walks with finite superexponential moment, we show that the equality is achieved if and only if the Gibbs density is equivalent to 
the hitting measure. As a corollary, if the action is not convex cocompact, any hitting measure is singular to any Gibbs density. 
\end{abstract}

\maketitle

\section{Introduction}


Let $M$ be a manifold of negative curvature. Then the 
boundary at infinity $\partial X$ of its universal cover $X = \widetilde{M}$ carries two types of measures (see e.g. \cite{Ledrappier-ICM}):

\begin{itemize}
    \item 
on the one hand, \emph{Gibbs measures} on the unit tangent bundle capture 
the asymptotic distribution of weighted periodic orbits for the geodesic flow. 
These include the Bowen-Margulis measure, the Liouville measure, and the harmonic measure associated to Brownian motion. 
To each of these can be associated a pair of measures on the boundary called Gibbs densities. For the Bowen-Margulis measure these conditionals are classical Patterson-Sullivan measures and for the Liouville measure they are in the Lebesgue measure class (see \cite{Sin}, \cite{BoR}, \cite{Mar}, \cite{Pat}, \cite{Sullivan}, \cite{PaPo}, \cite{Ham}, \cite{PPS}, \cite{BPP}, among others).



\item
on the other hand, one can run a random walk 
on the fundamental group $\Gamma$ of $M$, which acts by isometries on $X$. This determines a measure on the boundary $\partial X$ which is called the \emph{harmonic measure} or \emph{hitting measure}. 

\end{itemize}

In this paper, we will compare these two classes of measures when $M$ is a geometrically finite manifold of pinched negative curvature (and more generally, a geometrically finite quotient of a $CAT(-1)$ space). 


\medskip
Several numerical invariants have been introduced to capture the global dynamical and geometric properties of a random walk on a group. Namely, let $\mu$ be a probability measure on $\Gamma$, and define a random walk 
$$\omega_n := g_1 \dots g_n$$
where the $(g_i)$ are i.i.d. with distribution $\mu$. Let $o \in X$ be a basepoint. Then one defines: 

\begin{enumerate}
    \item the \emph{entropy} $h_\mu$, introduced by Avez \cite{Avez}, 
    $$h_\mu := \lim_{n \to \infty} \frac{1}{n} \sum_g - \mu^n(g) \log \mu^n(g);$$ 
    \item the \emph{drift} $\ell_\mu$ of the random walk 
    $$\ell_\mu := \lim_{n \to \infty} \frac{d(o, \omega_n o)}{n}$$
    where the limit exists a.s. and is independent of $\omega_n$ and $o$; 
    \item the \emph{critical exponent} $v$ of the action
    $$v := \limsup_{R \to \infty} \frac{1}{R} \log \# \{ g \in G \ : \ d(o, go) \leq R\}.$$
\end{enumerate}
These quantities are related via the inequality 
\begin{equation}\label{E:hlv}
h_\mu \leq \ell_\mu v
\end{equation}
due to Guivarc'h \cite{Guivarch} (see also Vershik \cite{Vershik}, who calls it the \emph{fundamental inequality}). Let us remark that $\frac{h}{l}$ is the Hausdorff dimension of the harmonic measure \cite{Tanaka}. On the other hand, to a random walk on $\Gamma$ we can associate a \emph{Green metric} $d_{G}$ on $\Gamma$,
defined in \cite{Blachere-Brofferio}, with $d_{G}(g,h)$ defined as the negative logarithm of the probability that a random path starting at $g$ ever hits $h$.
A measure $\mu$ is \emph{generating} $\Gamma$ if the semigroup generated by the support of $\mu$ equals $\Gamma$.

In order to define a Gibbs measure, one is given a \emph{potential}, i.e. a H\"older continuous, $\Gamma$-invariant function $F : T^1 X \to \mathbb{R}$.
Then one defines the \emph{topological pressure} as
$$v_F := \limsup_{n \to \infty} \frac{1}{n} \log \sum_{n - 1 \leq d(o, go) \leq n} e^{\int_o^{go} F}$$ 
and a Gibbs measure $m_F$ is a probability measure on $T^1 M$ whose pressure equals the topological pressure (if one exists). We may lift $m_F$ to a Radon measure $\widetilde{m}_F$ on $T^{1}X \cong (\partial X \times \partial X \setminus \Delta) \times \mathbb{R}$, and this defines a pair of measures on $\partial X$, known as \emph{Gibbs densities}. 

In order to account for the potential, we introduce a new notion of drift. We define the following distance (which we call \emph{F-ake distance}, as it does not satisfy any of the usual properties of a distance)
$$d_{F}(x,y) :=\int^{y}_{x}F \ dt,$$
where the integral is taken along the geodesic from $x$ to $y$. 
Then, we define the \emph{F-ake drift} as 
$$\ell_{F, \mu} := \lim_{n \to \infty} \frac{d_F(o, \omega_n o)}{n}$$
where, under suitable hypotheses, the limit exists almost surely and is constant. 

We show that hitting measures and Gibbs densities are in the same measure class if and only if we have a relation between the dynamical quantities defined above. Moreover, this holds if and only if the Green metric, the space metric, and the F-ake metric are related as follows. 

\begin{theorem}\label{intro:maintheoremGibbs}
Let $M$ be a geometrically finite manifold of pinched negative curvature and let $\Gamma := \pi_1(M)$. Let $\kappa_F$ be a Gibbs density for a H\"older potential $F$, 
and let $\nu_\mu$ be the hitting measure for a random walk driven by a measure $\mu$ generating $\Gamma$ with finite superexponential moment.
Then we have the inequality:
$$h_\mu  \leq \ell_\mu  v_F - \ell_{F, \mu}.$$ 
Moreover, the following conditions are equivalent.
\begin{enumerate}
    \item The equality
    $$h_\mu =\ell_\mu v_F-\ell_{F, \mu}$$ holds.
    \item The measures $\nu_{\mu}$ and $\kappa_{F}$ are in the same measure class.
    \item For any basepoint $o\in X$, 
    there exists $C\geq 0$ such that  $$|d_G(e,g)-v_Fd(o,go)+d_{F}(o,go)|\leq C$$
    for every $g \in \Gamma$.

\end{enumerate}
\end{theorem}

In fact, we do not need $X$ to be a manifold, as the result still holds when $X$ is a proper $CAT(-1)$ space (even though the H\"older condition becomes slightly more technical). See Theorem \ref{maintheoremGibbs}.

In the case when $\Gamma$ is not convex cocompact, it is not too hard to show that (3) cannot hold, 
yielding the following. 

\begin{corollary} \label{Gibbssingularity}
If the action is not convex cocompact, then
any Gibbs density $\kappa_F$ is mutually singular to any hitting measure $\nu_\mu$.
\end{corollary}

In particular, if the derivatives of the curvature are uniformly bounded, then the \emph{geometric potential} $F := - \frac{d}{dt}_{| t = 0} \log J^{su}$ 
(where $J^{su}$ is the Jacobian of the geodesic flow in the unstable direction) is H\"older continuous and the Gibbs measure $m_F$ is the Liouville measure, hence we obtain that no hitting measure is in the same class as the Lebesgue measure. 

Theorem \ref{intro:maintheoremGibbs} addresses a question of Paulin-Pollicott-Shapira (\cite{PPS}, page 9). 
Our results Theorem \ref{intro:maintheoremGibbs} and Corollary \ref{Gibbssingularity} are new even in the case where $F=0$, i.e. when $\kappa_F$ is the quasiconformal or Patterson-Sullivan measure. 
In this case, we need not assume $X$ is $CAT(-1)$: it need only be a proper geodesic Gromov hyperbolic space. We thus give the complete proof in this context in Theorem \ref{maintheorem}.

\medskip
When in addition $\Gamma \curvearrowright X$ is a convex cocompact action of a word hyperbolic group and the measure $\mu$ is symmetric, the latter result is proved by Blach\`ere-Ha\"issinsky-Mathieu in \cite{BlachereHassinskyMathieu2}. 
The authors there also prove that if $\Gamma \curvearrowright X$ is an action of a hyperbolic group which is not convex cocompact then the hitting and Patterson-Sullivan measures are singular.
In particular this is true for 
finite covolume Fuchsian groups with cusps, a fact also obtained by Guivarc'h-LeJan  \cite{Guivarch-winding}, Deroin-Kleptsyn-Navas
\cite{Deroin-Kleptsyn-Navas}, and Gadre-Maher-Tiozzo \cite{Gadre-Maher-Tiozzo}. Note that a lattice $\Gamma < \textup{Isom }\mathbb{H}^2$  is a hyperbolic group, while this is not true for lattices in $\textup{Isom }\mathbb{H}^n$, $n \geq 3$. 

For symmetric random walks on geometrically finite but not convex cocompact isometry groups of Gromov hyperbolic spaces, the singularity of harmonic and Patterson-Sullivan measures was obtained by Gekhtman-Gerasimov-Potyagailo-Yang \cite{GGPY}. Note these groups need not be hyperbolic.
As a corollary (Corollary \ref{singularcorollary}) of Theorem \ref{maintheorem}, we obtain the corresponding result for asymmetric random walks.

Note finally that for lattices in constant negative curvature the Patterson-Sullivan measure lies in the Lebesgue measure class, while this need not be the case in variable curvature.

\medskip

In the opposite direction, Connell-Muchnik (\cite{CM-Conformal}, \cite{CM-Gibbs}) show that for cocompact 
isometry groups of $CAT(-1)$ spaces any Gibbs state on the boundary of a $CAT(-1)$ space is a hitting measure for some random walk with finite first moment.

If one replaces the random walk by Brownian motion, Ledrappier \cite{Ledrappier-Israel} proved that harmonic measures coincide with the Patterson-Sullivan measure if and only if an analogue of the fundamental inequality is satisfied; moreover, in dimension $2$ these two measures coincide if and only if the curvature is constant. The corresponding question in higher dimensions is a well-known open problem.

\medskip
We note that the hitting measure $\nu$ is a conditional measure
for a geodesic flow invariant measure on $T^{1}X$, called the \emph{harmonic invariant measure} associated to the random walk (in analogy to the harmonic measure associated to Brownian motion). In turn, this measure induces a finite flow invariant measure on $T^1 M$. The construction is due to Kaimanovich \cite{Kaierg} for convex-cocompact manifolds and to Gekhtman-Gerasimov-Potyagailo-Yang \cite{GGPY} for geometrically finite ones. Moreover, axes of loxodromic elements associated to typical random walk trajectories equidistribute with respect to this measure (see \cite{Gekhtman}).

Our results give conditions for equivalence of a Gibbs measure and a harmonic invariant measure, and in particular they imply: 

\begin{corollary}
If $\Gamma$ is geometrically finite but not convex cocompact, then any harmonic invariant measure is singular with respect to any Gibbs measure on $T^1 M$. 
\end{corollary}

The Guivarc'h inequality has also been studied for word metrics. In this context, Gouezel, Matheus and Maucourant proved that the inequality (\ref{E:hlv}) is strict for any superexponential moment generating random walk on a word-hyperbolic group which is not virtually free. Dussaule-Gekhtman \cite{DG} extended this result to large classes of relatively hyperbolic groups, including finite covolume isometry groups of pinched negatively curved manifolds and geometrically finite Kleinian groups.

\subsection*{Acknowledgements}
We thank the Fields Institute for its support during the semester on ``Teichm\"uller theory and its connections to geometry, topology and dynamics". G.T. is partially supported by NSERC and the Alfred P. Sloan Foundation. 

\section{Background}

\subsection{Notation}
For two quantities $f$ and $g$ we write $f\asymp_{C}g$ if $\frac{1}{C}f\leq g\leq C f$ and $f\asymp_{+,C}g$ if $f-C\leq g\leq f+C$.
We write $f\asymp g$ if $f\asymp_{C} g$ for some constant $C$ and similarly for $f\asymp_{+} g$.
Also, whenever $f\leq C g$ (resp. $f\leq C+g$) for some constant $C$, we will use the notation $f\lesssim g$ (resp. $f\lesssim_{+}g$).

\subsection{Geometrically finite actions}

Let $(X,d)$ be a proper geodesic Gromov hyperbolic space and $\Gamma< \textup{Isom}(X,d)$ a properly discontinuous group of isometries.

The set  $\Lambda$ of accumulation points in the Gromov boundary $\partial X$ of any orbit $\Gamma x\ (x\in X)$ is called the  \emph{limit set} of the action $\Gamma \curvearrowright X$.
A point $\zeta \in \Lambda$ is called \emph{conical} if for every geodesic ray $\gamma$ converging to $\zeta$ and every $x\in X$ there is some $D>0$ such that  $\gamma$ has infinite intersection with the $D$-neighborhood of the orbit $\Gamma o$.
A point $\zeta \in \Lambda$ is called \emph{bounded parabolic} if its stabilizer $Stab(\zeta)$ in $\Gamma$ is infinite and acts cocompactly and properly discontinuously on $\Lambda \setminus \{x\}$.
The action  $\Gamma \curvearrowright X$ is said to be \emph{geometrically finite} if every point of $\Lambda$ is either conical or bounded parabolic.  

Let $\Gamma \curvearrowright X$ be a non-elementary geometrically finite action. 
Let $v=v_{\Gamma,X}$ be the \emph{critical exponent} of $\Gamma$ with respect to the action:
$$v :=\inf \left\{s:\sum_{g\in  G}e^{-s d(o, go)}<\infty \right\}=\liminf_{R\to \infty}\frac{1}{R}\log |\Gamma o \cap B_{R}(o)|.$$
We assume $v<\infty$.

\subsection{Busemann functions}

For $x,y,z\in X$ let us define the 
\emph{Busemann function} as
$$\beta_{z}(x,y):=d_{X}(x,z)-d_{X}({y},z)$$
and its extension to the boundary as 
$$\beta_{\xi}(x,y):=\liminf_{z\in X, z\to \xi}\beta_{z}(x,y)$$
where $\xi \in \partial X$.
Moreover, the \emph{Gromov product} of $x, y$ based at $z$ is
$$\rho_{z}(x,y):=\frac{ d_{X}(z,x)+d_{X}(z,y)-d_{X}(y,x)}{2}$$
and for $\xi, \zeta \in \partial X$ we define it as
$$\rho_{z}(\xi, \zeta):=\liminf_{\stackrel{x,y\in X}{x\to \xi, y\to \zeta}}\rho_{z}(x,y).$$

\subsection{Quasiconformal measures}

Fix a basepoint $o\in X$.
A probability measure $\kappa$ on the limit set $\Lambda \subset \partial X$ is called \emph{quasiconformal} of dimension $s$ for $\Gamma \curvearrowright X$ if
for any $g \in G$ and a.e. $\zeta \in \Lambda$  
 $$\frac{dg\kappa}{d\kappa}(\zeta)\asymp e^{s\beta_{\zeta}(o,go)},
 $$
where the implicit constant depends on the basepoint $o$ but not on $g$ and $\zeta$.
 
If the growth rate is $v<\infty$ there necessarily exists a quasiconformal measure of dimension $v$ \cite[Theorem 5.4]{Coornaert}. Moreover, quasiconformal measures of dimension $s<v$ do not exist \cite[Corollary 6.6]{Coornaert}, while quasiconformal measures of dimension $s>v$ give zero weight to conical limit points \cite[Proposition 2.12]{MYJ}, and hence are atomic.  If $\Gamma$ is of divergence type, a $v$-dimensional quasiconformal measure is unique up to bounded density, ergodic, and gives full weight to conical limit points \cite[Corollary 3.14]{MYJ}. 
Otherwise, any quasiconformal measure  gives zero weight to conical limit points \cite[Proposition 2.12]{MYJ}.  

We call a quasiconformal measure of dimension $v$ a \emph{Patterson-Sullivan measure}.

\subsection{Random walks and the Green metric}

Let $\mu$ be a probability measure on $\Gamma$. 
Assume that the support of $\mu$ generates $\Gamma$ as a semigroup.
Assume furthermore that $\mu$ has \emph{finite superexponential moment} with respect to some (equivalently every) word metric $\Vert .\Vert $ on $\Gamma$: that is,
$$\sum_{g\in \Gamma}c^{\Vert g\Vert } \mu(g)<\infty$$ for all $c>1$.

The Green function $G: \Gamma \times \Gamma \to \mathbb{R}$ associated to $(\Gamma,\mu)$ is defined to be the total weight $G(x,y)$ of all paths between $x$ and $y$. 
Letting $d_G(x, y):= - \log G(x,y)$ we obtain a (possibly asymmetric) metric on $\Gamma$, called the \emph{Green metric}.

Let $P$ be the measure on sample paths induced by $\mu$.
For $P$-almost every $\omega \in \Gamma^\mathbb{N}$ the quantities 
$$\ell_\mu := \lim_{n\to\infty} \frac{d(o, \omega_{n}o)}{n}$$ and 
$$h_\mu :=\lim \frac{-\log \mu^{*n}(\omega_n)}{n}$$ are defined and are independent of $\omega$.
They are called respectively the \emph{drift} and \emph{asymptotic entropy} of the random walk.



By \cite[Theorem 1.3]{GGPY}, conical limit points are in one-to-one correspondence with a subset of the Martin boundary, which is the horofunction boundary of $d_{G}$.
This means if $(g_n)$ is a sequence in $\Gamma$ converging to a conical limit point $\xi\in \partial X$, then
for every $g$, $\frac{G(g,g_n)}{G(e,g_n)}$ converges to some limit $K_{\xi}(g)$.
In particular, if $g,g'$ are fixed, then $\frac{G(g,g_n)}{G(g',g_n)}$ 
converges to $\frac{K_{\xi}(g)}{K_{\xi}(g')}$.
Define then
$$\beta^G_{\xi}(g,g') := d_G(g, g_n) - d_G(g', g_n) = -\log  \frac{K_{\xi}(g)}{K_{\xi}(g')}$$
which can be considered as Busemann functions for the Green metric.

Let $\nu$ be the unique $\mu$-stationary probability on $\partial X$.
The measure $\nu$ is necessarily ergodic, has no atoms, and is supported on conical limit points.
It satisfies a conformal-type property with respect to the Green metric:
\begin{equation} 
\label{E:Green-qc}
\frac{dg\nu}{d\nu}(\xi)=K_{\xi}(g)=e^{-\beta_\xi^{G}(g,e)}
\end{equation}
for any $g\in G$ and $\nu_{g}$-almost every conical limit point $\xi$,
see e.g. \cite[Theorem~24.10]{Woess}.

\subsection{Comparing shadows}

For $r>0$ and $x,y\in X$ the shadow $Sh_{r}(x,y)$ consists of all points $\zeta \in \partial X$ such that some geodesic ray from $x$ to $\zeta$ intersects $B_{r}(y).$
The following analogue of Sullivan's classical shadow lemma is due to Coornaert \cite{Coornaert}.

\begin{proposition} \label{PSshadow}
Let $\kappa$ be a quasiconformal measure for $\Gamma \curvearrowright X$.
For large enough $r>0$ we have $\kappa(Sh_{r}(o,go))\asymp e^{-vd(o,go)}$ where the implied constant depends only on $r$, $o$ and the quasiconformality constant.
\end{proposition}

Let $\Lambda \subset \partial X$ be the limit set of $\Gamma$.
For $D>0$ let $\Lambda_{D}\subset \Lambda$ consist of $\zeta$ such that any geodesic ray in $X$ converging to $\zeta$ intersects the $D$-neighborhood of $\Gamma o$ infinitely many times.
The set $\Lambda_D$ is $\Gamma$-invariant, and $\bigcup_{D>0}\Lambda_D$ is precisely the set of conical points of $\Lambda$.
Thus, we have: 

\begin{lemma}\label{conicalconstant}
Any $\Gamma$ quasi-invariant ergodic measure on $\partial X$ which gives full weight to conical limit points of $\Lambda$ gives full weight to $\Lambda_D$ for large enough $D$.
\end{lemma}

We will prove a shadow lemma for the $\mu$-harmonic measure $\nu$.
\begin{proposition}\label{harmonicshadow}
For large enough $r>0$ we have $\nu(Sh_{r}(o,go))\asymp e^{-d_G(e,g)}$ where the implied constant depends only on $r$ and $o$.
\end{proposition}

The following is a re-formulation of the deviation inequalities of Gekhtman-Gerasimov-Potyagailo-Yang \cite[Corollary 1.4]{GGPY}.
\begin{proposition}\label{relAncona}
For each $o\in X$ and $D>0$ there is an $A>0$ such that for all $g_{1},g_{2},g_{3}\in \Gamma$ such that $g_{2}o$ lies within distance $D$ of a geodesic $[g_{1}o,g_{3}o]$ we have 
$$d_{G}(g_{1},g_{2}) + d_{G}(g_{2},g_{3}) \leq d_{G}(g_{1},g_{3})+A.$$
\end{proposition}

\begin{proof}[Proof of Proposition \ref{harmonicshadow}]
Let $r>0$ be such that the complement of $\Lambda_r$ has $\nu$-measure zero.
Note we have by eq. \eqref{E:Green-qc}
$$\nu(Sh_{r}(o,go))=g^{-1}\nu(Sh_{r}(g^{-1}o,o)) =  \int_{ Sh_{r}(g^{-1}o,o)}e^{-\beta^{G}_{\zeta}(g^{-1},e)}d\nu(\zeta).$$
Consider a point $\zeta \in Sh_{r}(g^{-1}o,o)\cap \Lambda_r$.
By definition, any geodesic ray $[g^{-1}o,\zeta)$ in $X$ contains a point in $B_{r}(o)$ and there is a sequence $g_n\in \Gamma$ with $d(g_n o,o)\to \infty$ and $d(g_n o, [o,\zeta) )<r$.
Then by Proposition \ref{relAncona} we have for each $n$: $$d_{G}(g^{-1},e) -A \leq d_{G}(g^{-1},g_n)-d_{G}(e,g_n)$$ where $A$ depends only on $r$.
Taking limits 
as $g_n \to \zeta$ and by the triangle inequality we obtain $$d_{G}(g^{-1},e)-A\leq \beta^{G}_{\zeta}(g^{-1},e)\leq d_{G}(g^{-1},e).$$
Fix a metric $\rho$ on $\partial X$. Let $0<c<\textup{diam}(\Lambda, \rho)/100$. 
By \cite[Lemma 6.3]{Coornaert}, there is an $r_0>0$ such that for  any $r>r_0$ and $g\in \Gamma$ the complement $\partial X \setminus Sh_{r}(g^{-1}o,o)$ is contained in a $\rho$-ball of radius $c$. Consequently, $Sh_{r}(g^{-1}o,o)$ must contain a $\rho$-ball of radius $c$ centered at a point of $\Lambda$.
Since $\nu$ has full support on $\partial X$ there is a constant $t>0$ such that any such ball has $\nu$-measure at least $t$. 
  Consequently,  we have $1\geq \nu(Sh_{r}(g^{-1}o,o)) \geq t>0$ for all $g\in \Gamma$. 
This completes the proof.
\end{proof}

\section{A differentiation theorem}
Unlike in the hyperbolic group case, the harmonic measure $\nu$ is not known to be doubling for the visual metric on $\partial X$. See Tanaka's \cite[Question 4.1]{Tanaka} for a discussion.  However, we can still prove a Lebesgue differentiation-type theorem.
\begin{proposition}\label{localdensity}
Let $\nu$ be a $\Gamma$ quasi-invariant ergodic measure on $\partial X$ supported on conical points.
Assume furthermore that for a constant $C>0$ and all large enough $r$ we have 
$\nu(Sh_{2r}(o,go))\leq C \nu(Sh_{r}(o,go))$ for all $g \in G$.
Let $\kappa$ be any finite Borel measure on $\partial X$.
Then the following holds for large enough $r>0$. 
\begin{enumerate}[a)]

   \item If $\nu$ and $\kappa$ are mutually singular then for $\nu$-almost every $\xi \in \partial X$ we have $$\lim_{\stackrel{g\in \Gamma: g\to \xi}{\xi \in Sh_{r}(o,go)}} \frac{\kappa(Sh_{r}(o,go))}{\nu(Sh_{r}(o,go))}= 0.$$
   
   \item If $\nu$ and $\kappa$ are equivalent then for $\nu$-almost every $\xi \in \partial X$ we have $$\lim_{\stackrel{g\in \Gamma: g\to \xi}{\xi \in Sh_{r}(o,go)}} \frac{\log \kappa(Sh_{r}(o,go))}{\log \nu(Sh_{r}(o,go))}= 1.$$
\end{enumerate}
\end{proposition}

To prove this proposition we will need the notion of Vitali relation, which is a generalization of coverings by balls in doubling metric spaces. See Federer's book \cite[Sections 2.8 and 2.9]{Federer} for background on Vitali relations and their application to differentiation in metric spaces.
We were inspired by their use in \cite{MYJ} to study quasiconformal measures for divergence type groups of isometries of Gromov hyperbolic spaces.

Let $\Lambda$ be a metric space. A \emph{covering relation} is a subset of the set of all pairs $(\xi,S)$ such that $\xi \in S\subset \Lambda$. A covering relation  $\mathcal{C}$ is said to be \emph{fine} at $\xi \in \Lambda$ if there exists a sequence $S_n$ of subsets of $\Lambda$ with $(\xi,S_n)\in \mathcal{C}$ and such that  the diameter of $S_n$ converges to zero.

For a covering relation $\mathcal{C}$ and any measurable subset $E\subset \Lambda$, define $\mathcal{C}(E)$ to be the collection of subsets $S\subset \Lambda$ such that $(\xi,S)\in \mathcal{C}$ for some $\xi \in E$. 

A covering relation $\mathcal{V}$ is said to be a \emph{Vitali relation} for a finite measure $\mu$ on $\Lambda$ if it is fine at every point of $\Lambda$ and if the following holds: 
if $\mathcal{C}\subset \mathcal{V}$ is fine at every point of a measurable subset $E\subset \Lambda$ then $\mathcal{C}(E)$ has a countable disjoint
subfamily $\{S_n\} \subset \mathcal{C}(E)$ such that $\mu(E\setminus \bigcup^\infty_{n=1}S_{n})=0$.

For a covering relation $\mathcal{V}$ and a function $f$ on $\Lambda$ let us
 denote 
$$(\mathcal{V})\lim \sup_{S\to x}f := \lim_{\epsilon \to 0} \sup \{f(x):(x,S)\in \mathcal{V}, x\in S, \textup{diam}(S)<\epsilon\}.$$
Similarly we define $(\mathcal{V})\lim \inf_{S\to x}f$, and if the two limits are equal we denote its common value as $(\mathcal{V})\lim_{S\to x}f$.

The following criterion guarantees a covering relation is Vitali.

\begin{proposition}{\cite[Theorem 2.8.17]{Federer}}
\label{Vitalicriterion}
Let $\mathcal{V}$ be a covering relation on $\Lambda$ such that each
$S\in \mathcal{V}(\Lambda)$ 
 is a closed bounded subset and $\mathcal{V}$ is fine at every point of $\Lambda$.
Let $\lambda$ be a measure on $\Lambda$ such that $\lambda(A)>0$ for all $A\in \mathcal{V}(\Lambda)$.
For a positive
function $f$ on $\mathcal{V}(\Lambda)$, $S\in \mathcal{V}(\Lambda)$, and a constant $\tau>1$ define
$\widetilde{S}$ to be the union of all $S'\in \mathcal{V}(\Lambda)$ which have nonempty intersection with $S$ and satisfy $f(S')\leq \tau f(S)$.
Suppose that for $\lambda$-almost every $\xi\in \Lambda$ we have 
$$\lim \sup_{S\to \xi} \left( f(S)+\frac{\lambda(\widetilde{S})}{\lambda(S)} \right) <\infty.$$
Then the relation $\mathcal{V}$ is Vitali for $\lambda$.

\end{proposition}

Let $D$ be large enough so that the complement of $\Lambda_{D}$ has $\nu$ measure zero, and large enough for all $r\geq D$ to satisfy Proposition \ref{harmonicshadow}. 

\begin{lemma}
Define the covering relation 
$$\mathcal{V} := \{ (\xi,Sh_{2D}(o,go)\cap \Lambda_D) \textup{ where } \xi \in Sh_{2D}(o,go)\cap \Lambda_D\}. $$
Then $\mathcal{V}$ is a Vitali relation for $(\Lambda_D,\nu).$
\end{lemma} 

\begin{proof}
This is shown in \cite[Lemma 4.5]{MYJ} for quasiconformal measures, and the proof is essentially the same in our setting.
Indeed, by definition of $\Lambda_D$ this relation is fine at every point of $\Lambda_D$. Furthermore, by the thin triangles property $\tilde{Sh_{2D}(o,go)\cap \Lambda_D}$ is contained in $\tilde{Sh_{D'}(o,go)}$ for another constant $D'$ (see proof of \cite[Lemma 4.5]{MYJ}). Thus, letting $f(Sh_{2D}(o,go)\cap \Lambda_D):=e^{-vd(o,go)}$ and any $\tau>1$ we see using the fact that
$$\nu(Sh_{D'}(o,go))/\nu(Sh_{D}(o,go))\asymp 1$$ (implied by Proposition \ref{harmonicshadow}) that $\mathcal{V}$ satisfies the hypothesis of Proposition \ref{Vitalicriterion}. Hence it is a Vitali relation.
\end{proof}

The following is obtained by combining Theorems 2.9.5 and 2.9.7 of \cite{Federer}.

\begin{proposition}\label{RNexists}
Let $\Lambda$ be a metric space, $\kappa_1$ a finite Borel measure on $\Lambda$ and $\mathcal{V}$ a Vitali relation for $\kappa_1$.
Let $\kappa_2$ be any finite Borel measure on $\Lambda$. Define a new Borel measure $\psi(\kappa_{1},\kappa_{2})$ by 
$$\psi(\kappa_{1},\kappa_{2})(A) :=\inf \{\kappa_2(B): B \textrm{ Borel},  \kappa_1(B\Delta A)=0\}.$$
This measure is absolutely continuous with respect to $\kappa_1$.
The limit $$D(\kappa_1,\kappa_2,\mathcal{V},x):=(\mathcal{V})\lim_{S\to x} \frac{\kappa_2(S)}{\kappa_1(S)}$$ exists for $\kappa_1$-almost every $x$ and equals the Radon-Nikodym derivative
$\frac{d\psi(\kappa_{1},\kappa_{2})}{d\kappa_1}$.
\end{proposition}

As a corollary we obtain:
\begin{corollary} \label{Vitalidichotomy}
\begin{enumerate}[a)]
    \item If the $\kappa_i$ are mutually singular, then $$(\mathcal{V})\lim_{S\to \zeta}\frac{\kappa_{2}(S)}{\kappa_{1}(S)}=0$$   for $\kappa_{1}$-almost every $\zeta \in \Lambda$.
    
    \item If the $\kappa_{i}$ are equivalent and non-atomic, then 
    $$(\mathcal{V})\lim_{S\to \zeta} \frac{\log \kappa_{2}(S)}{\log \kappa_{1}(S)}=1$$
    for $\kappa_{1}$-almost every $\zeta \in \Lambda$.
\end{enumerate}
\end{corollary}

\begin{proof}
If $\kappa_1$ and $\kappa_2$ are mutually singular, then by definition 
$\psi(\kappa_{1},\kappa_{2})=0$.
Together with Proposition~\ref{RNexists}, this proves a).

For b), assume $\kappa_1$ and $\kappa_2$ are equivalent. 
By Proposition \ref{RNexists} we have 
$$D(\kappa_1,\kappa_2,\mathcal{V},\zeta)=(\mathcal{V})\lim_{S\to \zeta}\kappa_{2}(S)/\kappa_{1}(S)>0$$
for $\kappa_1$-almost every $\zeta$.
Thus, since $\kappa_1$ is non-atomic, we have
\begin{equation*}
   (\mathcal{V}) \lim_{S\to \zeta} \frac{\log \kappa_{2}(S)}{\log \kappa_{1}(S)}=1+
   (\mathcal{V}) \lim_{S\to \zeta} \frac{\log(\kappa_{2}(S)/\kappa_{1}(S))}{\log (\kappa_{1}(S))}= 1. \qedhere
\end{equation*}
\end{proof}

Applying this corollary to the Vitali relation $\mathcal{V}$ defined above and noting that $\nu(\Lambda^{c}_{D})=0$ completes the proof of Proposition \ref{localdensity}.

\section{Entropy and drift in hyperbolic spaces}

We will start with the proof of our main result
in the case $F = 0$, i.e. for the Patterson-Sullivan measure. In this case, we do not require the space $X$ to be $CAT(-1)$, but only $\delta$-hyperbolic. We prove the following.

\begin{theorem}\label{maintheorem}
Let $X$ be a $\delta$-hyperbolic, proper metric space, let $\Gamma$ be a geometrically finite group of isometries of $X$, and let $o\in X$ be a basepoint.

Let $\mu$ be a measure on $\Gamma$ with finite superexponential moment, let $\nu$ be its corresponding hitting measure, and let $\kappa$ be a Patterson-Sullivan measure on $\partial X$.

Then the following conditions are equivalent.
\begin{enumerate}
    \item The equality $h_\mu = \ell_\mu v$ holds.
    \item The measures $\nu$ and $\kappa$ are in the same measure class.
    \item The measures $\nu$ and $\kappa$ are in the same measure class with Radon-Nikodym derivatives bounded from above and below.
    \item There exists $C\geq 0$ such that for every $g \in \Gamma$, $$|d_G(e,g)-vd(o,go)|\leq C.$$
\end{enumerate}
\end{theorem}

\begin{corollary}\label{singularcorollary}
If $\Gamma\curvearrowright X$ is not convex cocompact then $\kappa$ and $\nu$ are mutually singular.
\end{corollary}
\begin{proof}
By Theorem \ref{maintheorem}, if $\kappa$ and $\nu$ are mutually singular then $d_G$ and $d$ are quasi-isometric. On the other hand, the Green metric is quasi-isometric to the word metric for any random walk on a non-amenable group with finite exponential moment  (see Proposition \ref{Greendecay} in the appendix). Thus, the orbit map from the Cayley graph to $X$ must be a quasi-isometry, which is impossible in the presence of parabolics \cite{Swenson}. 
\end{proof}

Let $\nu$ be the $\mu$-stationary measure on $\partial X$ and let $\kappa$ be a $v$-dimensional quasiconformal measure. 
In this section we prove that $h=lv$ if and only if $\nu$ and $\kappa$ are mutually absolutely continuous.
Let $r>0$ be large enough to satisfy Proposition \ref{harmonicshadow}, \ref{PSshadow}, and \ref{localdensity}.

For a sample path $\omega$ let $\omega_n$ be its $n$-th position.
Define then
$$\phi_n=\phi_n(\omega):=\frac{\kappa(Sh_{r}(o,\omega_n o))}{\nu(Sh_{r}(o, \omega_n o))}.$$
Let $\psi_n :=\log \phi_n$.
Notice that the expectation of $\phi_n$ is given by
$$E(\phi_n)=\sum_{g\in \Gamma}\mu^{*n}(g)\frac{\kappa(Sh_{r}(o,go))}{\nu(Sh_{r}(o,go))}.$$

\begin{proposition}\label{lemma4.14Haissinsky}
There exists $C_1>0$ such that for any $N\geq 1$ we have 
$$\frac{1}{N}\sum^{N}_{n=1}E(\phi_n)\leq C_1.$$
\end{proposition}

\begin{proof}
Consider $n,N$ with $1\leq n\leq N$. We will first show that there is some $k>0$ such that the quantity
$$R_{k} :=\sum_{g\in \Gamma: d(o,go)\geq k N} \frac{\kappa(Sh_{r}(o,go))}{\nu(Sh_{r}(o,go))}\mu^{*n}(g)$$ is bounded independently of $n,N$.

Let $\Vert .\Vert $ be any word norm on $\Gamma$.
By the shadow lemma for harmonic measure (Proposition \ref{harmonicshadow})
\begin{equation}\label{equationnu}
   \nu(Sh_{r}(o,go))\asymp G(e,g)=\sum^{\infty}_{n=1}\mu^{*n}(g).
\end{equation}
Furthermore, since the Green distance and word metric are quasi-isometric (see for example \cite[Lemma~4.2]{Haissinsky}), and $\kappa$ is a finite measure,
$$\frac{\kappa(Sh_{r}(o,go))}{\nu(Sh_{r}(o,go))}\lesssim e^{c\|g\|}$$ for a constant $c$. 
Also, we have $d(o,go)\leq t\Vert g\Vert $ for a constant $t >1$.
We obtain:
$$R_{k}\lesssim \sum_{g\in \Gamma: \|g\|\geq t^{-1}k N} \frac{\kappa(Sh_{r}(o,go))}{\nu(Sh_{r}(o,go))}\mu^{*n}(g) \leq $$
$$\leq \sum_{m\geq t^{-1}kN} e^{cm} \sum_{g: \|g\|=m}\mu^{*n}(g)\leq \sum_{m\geq t^{-1}kN} e^{cm}P(\|\omega_n\|\geq m).$$
Since $\mu$ has finite superexponential moment, we can apply the exponential Chebyshev inequality with exponent $2c$ to obtain
$$R_{k}\lesssim \sum_{m\geq t^{-1}kN} e^{-cm}E(e^{2c\|\omega_n\|}).$$
Since $n\leq N$ we have $$\|\omega_n\|\leq \sum^{N-1}_{j=0}\|\omega^{-1}_{j}\omega_{j+1}\|$$
from which we obtain, since the $\omega^{-1}_{j}\omega_{j+1}$ are independent random variables,
$$E(e^{2c\|\omega_n\|})\leq E_0^{N}$$ where $E_0=\sum_{g\in \Gamma}e^{2c\|g\|}\mu(g)$.
Choosing $k\geq \frac{t}{c}\log E_0$ we thus obtain 
$$R_{k}\lesssim \sum_{m\geq t^{-1}kN} e^{-cm}E(e^{2c\|\omega_n\|})\lesssim e^{-c t^{-1}kN}E_0^{N}\leq 1$$ giving us the desired estimate for $R_{k}$.

Now, we will show that the quantity
$$P_{N} :=\frac{1}{N}\sum^{N}_{n=1}\sum_{g\in B_{kN}}\frac{\kappa(Sh_{r}(o,go))}{\nu(Sh_{r}(o,go))}\mu^{*n}(g),$$ 
where $B_{kN} := \{ g \in \Gamma \ : \ d(o, go) \leq kN\}$, is bounded independently of $N$.
Together with the estimate on $R_k$ this will prove the proposition.
Interchanging the order of summation we get, using~(\ref{equationnu}),
$$P_{N}=\frac{1}{N}\sum_{g\in B_{kN}} \frac{\sum^{N}_{n=1}\mu^{*n}(g)}{\nu(Sh_{r}(o,go))}\kappa(Sh_{r}(o,go))
\lesssim \frac{1}{N}\sum_{g\in B_{kN}} \kappa(Sh_{r}(o,go)).$$
By \cite[Theorem 1.9]{Yang}, for any $a>0$, we have
$$|\{g\in \Gamma: n-a<d(o,go)\leq n\}|\lesssim e^{vn}.$$
Consequently, if we denote $A_{r, R} := \{ g \in \Gamma \ : \ r \leq d(o, go) \leq R\}$, we get

\begin{align*}
\sum_{g \in B_{kN} }
    \kappa(Sh_{r}(o,go))&=\sum^{N}_{n=1}\sum_{A_{k(n-1), kn}}
    \kappa(Sh_{r}(o,go))\\
    & \textup{and by Proposition \ref{PSshadow}} \\
    &\lesssim \sum^{N}_{n=1}
    \sum_{A_{k(n-1), kn}} e^{-vkn}\\
    & \lesssim \sum^{N}_{n=1} e^{vkn}e^{-vkn}\lesssim N.
\end{align*}
The estimate for $P_N$ follows.

\end{proof}

The following will be proved in the appendix.


\begin{proposition}\label{closetogeodesics}
There exists $C>0$ such that for each $k$, $k \leq n$ and $a>1$ we have
$$P(d(\omega_k o,[o,\omega_n o])>a) \leq Ce^{-a/C}$$
and
$$P(d(\omega_k o,[o,\omega_\infty))>a) \leq Ce^{-a/C}$$
where $[o,\omega_n o]$ and $[o,\omega_\infty)$ are any geodesics connecting the respective endpoints.
\end{proposition}

We now deduce the following proposition.

\begin{proposition}\label{lemma4.15Haissinsky}
There exists $C_2>0$ such that the sequence $E(\psi_n)+C_2$ is sub-additive and $\psi_{n}/n$ converges to $h-lv$ almost surely and in expectation. 
\end{proposition}
\begin{proof}
By the shadow lemmas (see Proposition~\ref{PSshadow} and Proposition~\ref{harmonicshadow}),
$$\frac{\psi_n}{n}=\frac{d_{G}(\omega_n,e)}{n} - \frac{d(o,\omega_n o)}{n} +O(1/n).$$
According to \cite[Theorem~1.1]{BlachereHassinskyMathieu1}, the term $\frac{1}{n}d_G(e,\omega_n)$ almost surely converges to $h$ whenever $\mu$ has finite entropy $h_\mu $, which is implied by finite first moment.
In other words, entropy is equal to the drift of $d_G$.
Thus, 
$\psi_{n}/n$ converges to $h-lv$ almost surely and in expectation.

Let $m,n\geq 1$.
The shadow lemmas for the Patterson-Sullivan measure (Proposition \ref{PSshadow}) and for the harmonic measure (Proposition \ref{harmonicshadow}) yield
$$\psi_n = \log \kappa(Sh_r(o, \omega_n o)) - \log \nu(Sh_r(o, \omega_n o)) 
 = - v d(o, \omega_n o) + d_G(o, \omega_n o) + O(1)$$
and the triangle inequality for $d_G$ implies that
\begin{align*}
& E(\psi_{n+m})- E(\psi_{n}) - E(\psi_{m}) \leq v E(d(o, \omega_n o)+ d(o, \omega_m o )-d(o, \omega_{n+m} o)) + O(1) = \\
& \textup{and by shift-invariance on the space of increments and $\delta$-hyperbolicity,} \\
&   =v E(d(o, \omega_n o)+ d(\omega_{n}o,\omega_{n+m}o)-d(o, \omega_{n+m} o))\leq 
v E(2d(\omega_{n}o,[o,\omega_{m+n}o])) + O(1).
\end{align*}
Lemma~\ref{closetogeodesics} implies that the last expression is bounded by a constant $C_2$, independent of $n$ and $m$, so that $E(\psi_n)+C_2$ is sub-additive.
\end{proof}

\begin{proposition}\label{lemma4.16Haissinsky}
Let $r$ be large enough for the conclusion of Proposition \ref{harmonicshadow} to hold. Then
\begin{enumerate}[a)] 
    \item If $\kappa$ and $\nu$ are not equivalent, then $\phi_n$ tends to $0$ in probability.
    \item If $\kappa$ and $\nu$ are equivalent, then $\frac{\log \kappa(Sh_{r}(o,\omega_n o))}{\log \nu(Sh_{r}(o,\omega_n o))}$ tends to 1 in probability.
\end{enumerate}
\end{proposition}
\begin{remark}
The only properties of $\kappa$ used are that $\Gamma$ preserves its measure class and acts ergodically on $(\partial X, \kappa)$.
\end{remark}
\begin{proof}

Recall, $\nu$ is always ergodic with respect to the action of $\Gamma$ on $\partial X$ and gives full weight to conical points.
On the other hand, $\kappa$ is ergodic, gives full weight to conical points when $\Gamma$ is divergence type and gives full weight to parabolic points when $\Gamma$ is convergence type. 
Thus, in either case, if the two measures are not equivalent, they are mutually singular.
The result now follows by combining Proposition \ref{localdensity} and Proposition \ref{closetogeodesics}. We give the details below.

Let $\alpha,c>0$.
By Proposition \ref{closetogeodesics} we have $P(\omega_\infty \notin Sh_{D}(o,\omega_n o))\leq F(D)$ independently of $n$ where $F(D)\to 0$ as $D\to \infty$.
Fix $D$ so that
$$P(\omega_\infty \notin Sh_{D}(o,\omega_n o))\leq \alpha$$
for all $n$.
By Proposition~\ref{localdensity}~a) we have, for $\nu$-almost every $\xi$,
$$\lim_{t\to \infty}\sup_{
\stackrel{d(o,go)\geq t}{\xi\in Sh_{D}(o,g o)}}\frac{\kappa(Sh_{D}(o,g o))}{\nu(Sh_{D}(o,g o))}= 0.$$
The shadow lemma for $\nu$ (Proposition~\ref{harmonicshadow}) shows that $\nu(Sh_{D}(o,g o))\leq C \nu(Sh_{r}(o,g o))$ where $C$ depends only on $r,D$.
Thus the quantity 
$$R_{t}(\xi) :=\sup_{
\stackrel{g:\Vert g\Vert  \geq t}{\xi\in Sh_{D}(o,go)}}\frac{\kappa(Sh_{r}(o,go))}{\nu(Sh_{r}(o,go))}$$  converges to $0$ as $t\to \infty$ for $\nu$-almost every $\xi$.
Furthermore, for almost every sample path $\omega$, we have 
$\Vert \omega_n\Vert \to \infty$. 
Thus, by Egorov's theorem, we may choose a subset $E\subset \Gamma^{\mathbb{N}}$ of sample paths with $P(E^c)<\alpha$ and such that $\Vert \omega_n\Vert \to \infty$ and $R_{n}(\omega_\infty)\to 0$ uniformly over $\omega\in E$.
It follows that $R_{\Vert \omega_n\Vert }(\omega_\infty)\to 0$ uniformly over $\omega\in E$. This means that for large enough $n$ (depending on $c$), the conditions $\omega\in E$ and $\frac{\kappa(Sh_{r}(o,\omega_n o))}{\nu(Sh_{r}(o,\omega_n o))}\geq c>0$ imply $\omega_\infty \notin Sh_{D}(o,\omega_n o)$. The latter has probability at most $\alpha$, so we get
$$P(\phi_n\geq c)\leq P(E^c)+P(\omega_\infty \notin Sh_{D}(o,\omega_n o))\leq 2\alpha.$$
As $\alpha,c>0$ were chosen arbitrarily we get $P(\phi_n\geq c)\to 0$ as $n\to \infty$ for each $c>0$ so $\phi_n \to 0$ in probability.

\item b) This time, we define for each $D>0$
$$R_{t}(\xi) :=
\sup_{
\stackrel{d(o,go)\geq t}{\xi\in Sh_{D}(o,go)}}\left |\frac{\log \kappa(Sh_{r}(o,go))}{\log \nu(Sh_{r}(o,go))}-1\right|.$$
Using b) of Proposition~\ref{localdensity} we obtain that for each $D$, 
$\lim_{t\to \infty}R_{t}(\xi)= 1$ for $\nu$-almost every $\xi$.
The proof is then similar to a). \qedhere

\end{proof}

We are now ready to prove the following.
\begin{theorem}\label{(1)implies(2)} 
The measures $\kappa$ and $\nu$ are equivalent if and only if $h=lv$. 
\end{theorem}
\begin{proof}
Assume that the Patterson-Sullivan and the harmonic measures are not equivalent.
Let $\beta>0$.
Let $A_n$ be the event $\{\phi_n\geq \beta\}$ and $B_n=A_n^c$.
For every $n$,
$$E(\psi_n)=\int_{A_n}\psi_ndP+\int_{B_n}\psi_ndP.$$
According to Proposition~\ref{lemma4.16Haissinsky}, $\phi_n$ converges to 0 in probability. Thus, there exists $n_0$ such that for every $n\geq n_0$, $P(B_n)\geq 1-\beta$. In particular,
$$\int_{B_n}\psi_ndP\leq P(B_n)\log \beta \leq (1-\beta)\log \beta.$$
Let $C_1$ be the constant in Proposition~\ref{lemma4.14Haissinsky}.
Jensen's inequality shows that $$\int_{A_n}\psi_n dP=P(A_n)\int_{A_n}\log \phi_n \frac{dP}{P(A_n)}\leq P(A_n)\log \left (\int_{A_n}\phi_n \frac{dP}{P(A_n)}\right ).$$
Rewrite the right-hand side as
$$P(A_n)\log \left (\int_{A_n}\phi_n \frac{dP}{P(A_n)}\right )=-P(A_n)\log P(A_n)+P(A_n)\log \left (\int_{A_n}\phi_n dP\right ).$$
The function $x\mapsto x\log x$ is first decreasing then increasing on $[0,1]$, so if $\beta$ is small enough,
$-P(A_n)\log P(A_n)\leq -\beta \log \beta$.
Moreover,
$$P(A_n)\log \left (\int_{A_n}\phi_n dP\right )\leq \beta \sup \left (0,\log E(\phi_n)\right ).$$
We thus have $$\int_{A_n}\psi_ndP\leq \beta\log \frac{1}{\beta}+\beta \sup \left (0,\log E(\phi_n)\right ).$$
According to Proposition~\ref{lemma4.14Haissinsky}, $\liminf E(\phi_n)\leq 2C_1$, so that there exists $p\geq n_0$ such that
$E(\phi_p)\leq 2C_1$.
In particular, for every small enough $\beta$, we can find $p$ such that
$$E(\psi_p)\leq (1-\beta)\log \beta+\beta\log \frac{1}{\beta}+\beta|\log 2C_1|.$$
The right-hand side tends to $-\infty$ when $\beta$ goes to 0.
If $\beta$ is small enough, we thus have for some $p$
$$E(\psi_p)+C_2\leq -1,$$
where $C_2$ is the constant in Proposition~\ref{lemma4.15Haissinsky}.
Since $E(\psi_p)+C_2$ is sub-additive, we have
$$\frac{1}{k}(E(\psi_{kp})+C_2)\leq E(\psi_p)+C_2\leq -1.$$
Finally, $\frac{1}{kp}E(\psi_{kp})$ converges to $h_\mu - \ell_\mu v$, so letting $k$ tend to infinity, we get
$$h_\mu - \ell_\mu v\leq -\frac{1}{p}<0.$$
Thus, $h_\mu < \ell_\mu v$.

Conversely, suppose the measures are equivalent.
By the shadow lemma for the Patterson-Sullivan measure $\kappa$ we have
$$\frac{-\log \kappa(Sh_{r}(o,go))}{d(o,go)}\to v$$
as $d(o,go)\to \infty$ and in particular 
$$\frac{-\log \kappa(Sh_{r}(o,\omega_n o))}{d(o,\omega_n o)}\to v$$
for almost every sample path.
Furthermore, for a.e. sample path we have
$$\lim_{n\to \infty}\frac{-\log \nu(Sh_{r}(o,\omega_n o))}{d(o,\omega_n o)}=\lim_{n\to \infty}\frac{d_{G}(e,\omega_n)}{d(o,\omega_n o)}=\lim_{n\to \infty}\frac{d_{G}(e,\omega_n)/n}{d(o,\omega_n o)/n}=\frac{h_\mu}{\ell_\mu}.$$
Thus, almost surely, $$\frac{\log \kappa(Sh_{r}(o,\omega_n o))}{\log \nu(Sh_{r}(o,\omega_n o))}\to \frac{v \ell_\mu }{h_\mu}.$$
As the measures are equivalent we have by Proposition \ref{localdensity} b) 
$$\frac{\log \kappa(Sh_{r}(o,\omega_n o))}{\log \nu(Sh_{r}(o,\omega_n o))}\to 1$$ almost surely, which ensures that $h_\mu = \ell_\mu v$.
\end{proof}

The following result is a consequence of Proposition~\ref{(1)implies(2)}.
Indeed, notice that $h$ and $l$ are the same for the measure $\mu$ and the reflected measure $\widecheck{\mu}$.

\begin{corollary}\label{reflectedequivalence}
Let $\widecheck{\nu}$ be the harmonic measure for the reflected random walk $\widecheck{\mu}$. Then $\widecheck{\nu}$ is equivalent to $\kappa$ whenever $\nu$ is equivalent to $\kappa$.
\end{corollary}

\section{Equivalence of measures and equivalence of metrics}
In this section we prove the following.
\begin{proposition}\label{(2)implies(3)}

If the harmonic measure $\nu$ and its reflection $\check{\nu}$ are both equivalent to the Patterson-Sullivan measure $\kappa$, then the Radon-Nikodym derivative $\frac{d\kappa}{d\nu}$ is bounded away from 0 and infinity.
\end{proposition}
This will use the following general lemma.

\begin{lemma}\label{lemmaergodicdouble}
Let $Z$ be a compact metrizable space and let $G$ act by homeomorphisms on $Z$. 
Let $\nu_1,\nu_2,\kappa_1,\kappa_2$ be Borel probability measures with full support on $Z$ and with $\nu_i$ equivalent to $\kappa_i$ for $i=1,2$.

Assume $G$ preserves the measure class of $\nu_i$  for $i=1,2$ and acts ergodically on $(Z\times Z, \nu_1 \otimes \nu_2)$ and $(Z\times Z, \kappa_1 \otimes \kappa_2)$.
Suppose there are positive, bounded away from 0, measurable functions $f_\nu,f_\kappa:Z\times Z\setminus \Diag \to \mathbb{R}$, bounded on compact subsets of $Z\times Z \setminus \Diag$ such that
$m_\nu=f_\nu \nu_1 \otimes \nu_2$ and $m_\kappa=f_\kappa \kappa_1 \otimes \kappa_2$ are $G$-invariant ergodic Radon measures on 
$Z\times Z\setminus \Diag $.
Then for each $i=1,2$, $\frac{d\nu_i}{d\kappa_i}$ is bounded away from $0$ and $\infty$.
\end{lemma}

\begin{proof}
Since $\nu_i$ and $\kappa_i$ are equivalent, we have $d\nu_i=J_i d\kappa_i$ for a measurable positive function $J_i$. We want to show $J_i$ is $\kappa_i$-essentially bounded. 

Since $m_\nu$ and $m_\kappa$ are $G$-invariant ergodic measures, either they are mutually singular or they are scalar multiples of each other. Thus, the assumption $d\nu_i=J_i d\kappa_i$ implies they are scalar multiples of each other.
Without loss of generality, we can assume that they coincide.
Note $dm_{\nu}(a,b)=J_{1}(a)J_{2}(b)f_{\nu}(a,b)d\kappa_1(a)d\kappa_2(b)$ so we have
$J_{1}(a)J_{2}(b)=f_{\kappa}(a,b)/f_{\nu}(a,b)$ for  $\nu_1\otimes \nu_2$-almost all $(a,b)$. 

Let $U,V$ be disjoint closed subsets in $Z$ with nonempty interior.
There is a $p\in V$ such that $J_{1}(a)J_{2}(p)=f_{\kappa}(a,p)/f_{\nu}(a,p)$ for $\nu_1$-almost all $a\in U$.
Dividing and noting that the $f_\nu$ and $f_\kappa$  are positive and bounded away from $0$ and infinity on $U\times V$, we see that $C^{-1}_U \leq J_{1}(a)/J_{1}(a') \leq C_U$ for $\nu_1$-almost all $a,a'\in U$. Thus, $J_1$ is $\nu_1$-essentially bounded on any closed subset $U$ whose complement has nonempty interior. Covering $Z$ by two such sets, we see that $J_1$ is essentially bounded. The same argument applies to $J_2$.
\end{proof}

\begin{proof}[Proof of Proposition \ref{(2)implies(3)}]
The $\Gamma$-action on $\nu \times \nu$
is ergodic (see \cite[Theorem~6.3]{Kaimanovich}) and since $\nu$ and $\widecheck{\nu}$ are both equivalent to $\kappa$, the $\Gamma$-action is also ergodic for $\kappa \otimes \kappa$.

To complete the proof of Proposition \ref{(2)implies(3)} we just need to show that $\kappa \otimes \kappa$ and $\nu \otimes \widecheck{\nu}$ can both be scaled by functions $f_\kappa$ and $f_\nu$ to obtain $\Gamma$-invariant Radon measures on $\partial^{2} X = (\partial X \times \partial X) \setminus \Diag$. 

For the harmonic measure $\nu$, we may take $f_\nu$ to be the \emph{Naim kernel}
defined for distinct conical points $\zeta,\xi$ as
$$\Theta(\zeta,\xi) :=\liminf_{g\in \Gamma \to \zeta} \lim_{h\in \Gamma \to \xi}\frac{G(g,h)}{G(e,g)G(e,h)}=\liminf_{g\to \zeta}\frac{K_{\xi}(g)}{G(e,g)}.$$
The construction is done in \cite[Corollary~10.3]{GGPY}.

For the Patterson-Sullivan measure $\kappa$, we may
define a measure $m'$ on $(\partial X \times \partial X)\setminus \Diag$ by 
$$dm'(\zeta, \xi) := e^{2v \rho^{X}_{o}(\zeta,\xi)}\ d\kappa(\zeta)\ d\kappa(\xi)$$
where we recall $\rho_o^X$ is the Gromov product. 
By \cite[Corollary 9.4]{Coornaert} this measure is $\Gamma$ quasi-invariant with uniformly bounded Radon-Nikodym cocycle.  Hence, by a general fact in ergodic theory the Radon-Nikodym cocycle is also a coboundary (see \cite{Furman}, Proposition 1). Thus, there exists a $\Gamma$-invariant measure $m$ on $(\partial X \times \partial X)\setminus \Diag$ in the same measure class as $m'$. In other words, one can take $f_\kappa$ to be within a bounded multiplicative constant of $e^{2v \rho_{e}(\zeta,\xi)}$. 
This completes the proof of Proposition \ref{(2)implies(3)}.
\end{proof}

We are now ready to prove:

\begin{proposition}\label{(3)implies(4)}
If the harmonic measure and the Patterson-Sullivan measure are equivalent, then $|d_{G}(g,g')-vd(go,g'o)|$ is uniformly bounded independently of $g,g'\in \Gamma$.
\end{proposition}

\begin{proof}
It follows from Proposition~\ref{(2)implies(3)} and Corollary \ref{reflectedequivalence} that if $\kappa$ and $\nu$ are equivalent, their respective Radon-Nikodym derivatives are bounded away from 0 and infinity.
In particular, for any Borel set $A\subset \partial X$ we have
$C^{-1} \nu(A)\leq \kappa(A)\leq C \nu(A)$, for a constant $C>0$.
The shadow lemmas for the Patterson-Sullivan and the harmonic measures show that $\kappa(Sh_{r}(o,go))\asymp \mathrm{e}^{-vd(o,go)}$ and $\nu(Sh_{r}(o,go))\asymp \mathrm{e}^{-d_G(e,g)}$.

It follows that $|d_G(e,g)-vd(o,go)|\leq C$ for some uniform $C$.
Since both distances are invariant by left multiplication, we have $|d_G(g,g')-vd(go,g'o)|\leq C$ for any $g,g'$.
\end{proof}

\section{Gibbs Measures}

Let us now generalize the previous results to Gibbs measures. We begin by stating the relevant definitions.

Let $X$ be a proper, geodesically complete $CAT(-1)$ space and let $T^{1}X$ be its unit tangent bundle.
Let $\pi:T^{1}X\to X$ be the projection map.
Let $\Gamma< \textup{Isom}(X)$ be a nonelementary group of isometries.
Let $F:T^{1}X\to \mathbb{R}$ be a $\Gamma$-invariant function, called a \emph{potential}.
Let $\iota:T^{1}X\to T^{1}X$ be the direction reversing involution.
For a potential $F$ let $\widecheck{F}=F\circ \iota$.
The following definition is from \cite[Definition 3.4]{BPP}.

\begin{definition}\label{HC}
The potential $F:T^{1}X\to \mathbb{R}$ satisfies the H\"older-control (HC) property if:
\begin{itemize}
\item[(a)]
There exists $c_1>0$ and $c_2\in (0,1)$ such that for all $x,y,x', y'\in X$ with $d(x,x'),d(y,y')\leq 1$ we have
$$\left| \int^{y'}_{x'}F dt - \int^{y}_{x}F dt \right|\leq (c_{1}+\max_{\pi^{-1}(B(x,d(x,x')))}|F|)d(x,x')^{c_2}+ (c_{1}+\max_{\pi^{-1}(B(y,d(y,y')))}|F|)d(y,y')^{c_2}$$

\item[(b)]
The potential $F$ has subexponential growth: for each $a>1$ there is a $b>0$ such that
$|F(x)-F(y)|\leq b a^{d(\pi(x),\pi(y))}$. 
\end{itemize}
\end{definition}

The HC property is satisfied, for instance, by any H\"older potential when $X$ is a contractible  manifold of pinched negative curvature \cite[Proposition 3.5]{BPP}. 
From now on, $F$ will be assumed to satisfy the HC property.

Define the \emph{F-ake metric} as 
$$d_{F}(x,y) :=\int^{y}_{x}F \ dt$$
(where the integral is taken along the geodesic from $x$ to $y$). We now define the \emph{topological pressure} of $F$ as 
$$v_F := \limsup_{n \to \infty} \frac{1}{n}\log \sum_{g\in S_n} e^{d_{F}(o,go)}.$$
where $S_n := \{ g \in \Gamma \ : \ n-1 \leq d(o, go) \leq n \}$.
We assume $v_F < \infty$.
Given a boundary point $\zeta \in \partial X$, let us define the \emph{Gibbs cocyle}\footnote{The comparison with \cite{BPP} is given by the formula $C_\xi(x,y) = - \beta_\xi^F(x,y) + v_F \beta_\xi(x, y)$.} as 
$$\beta^{F}_{\zeta}(x,y) := \lim_{z\to \zeta} \left( d_{F}(x,z)-d_{F}(y,z) \right).$$ 
The limit exists by \cite[Proposition 3.10]{BPP}. Note that by definition we have the cocycle property
$$\beta_\zeta^F(x, z) = \beta_\zeta^F(x, y) + \beta_\zeta^F(y, z)$$
for any $x, y, z \in X$, $\zeta \in \partial X$. 
Fix a basepoint $o\in X$.
The associated \emph{Gibbs density} is a $\Gamma$ quasi-invariant probability measure $\kappa_F$ on $\partial X$ such that \begin{equation}\label{E:quasi-kappa}
\frac{dg\kappa_{F}}{d\kappa_F}(\zeta)= \exp(-\beta^{F}_{\zeta}(o, go)+v_F\beta_\zeta(o, go)).
\end{equation}

The unit tangent bundle of $X$ is defined as $T^1 X := (\partial X \times \partial X \setminus \Delta) \times \mathbb{R}$ and on it there is a natural action of the geodesic flow. When there exists a flow invariant probability measure $m_F$ on $T^1 M := T^{1}X/\Gamma$ realizing the topological pressure $v_F$, its lift $\widetilde{m}_F$ to $T^{1}X$ is the unique, up to scaling, flow invariant measure equivalent to $\kappa_{F}\times \kappa_{\widecheck{F}}$. 
On the other hand, when  the latter construction projects to an infinite measure on $T^{1}M$, no finite measure on $T^{1}M$ realizing the topological pressure exists.
The pair $(\Gamma, F)$ is said to be of \emph{divergence type} if the series $$Q_{\Gamma,F}(s) :=\sum_{g\in \Gamma}e^{\int^{go}_{o}(F-s)}$$ diverges at its critical exponent $s=v_F$. In that case, there is a unique Gibbs density $\kappa_F$, and it is obtained as the weak limit  of measures $Q_{\Gamma,F}(s)^{-1}\sum_{g\in \Gamma}e^{\int^{go}_{o}(F-s)}\delta_{go}$ as $s\to v_{F}$. This is in particular the case when there exists a probability measure $m_F$ on $T^1 M$ realizing the topological pressure $v_F$. 
See \cite{BPP}, \cite{PPS} for details. 

When $(\Gamma,F)$ is of divergence type, $\kappa_{F}$ is $\Gamma$ ergodic and gives full weight to conical limit points. Otherwise, $\kappa_{F}$ gives zero weight to conical limit points \cite[Theorem 4.5]{BPP}.

Let us start with a few consequences of the HC property. 

\begin{lemma}
If $F$ satisfies the (HC) property then for all  and $x,y,z \in X$ we have 
$$|d_{F}(x,z)-d_{F}(y,z)|\leq (c_{1}+\max_{\pi^{-1}(B(x,d(x,y)))}|F|)(d(x,y)+1)$$ and
$$|d_{F}(z,x)-d_{F}(z,y)|\leq (c_{1}+\max_{\pi^{-1}(B(x,d(x,y)))}|F|)(d(x,y)+1).$$
\end{lemma}\label{(HC)Corollary}
\begin{proof}
Let $N := \lceil d(x,y) \rceil$ and pick $p_0, p_1, \dots, p_{N}$ points on $[x, y]$ with  $p_0 = x$, $p_N = y$ and $d(p_i, p_{i+1}) \leq 1$ for $0 \leq i \leq N -1$.
Then by Definition \ref{HC} a) we have $|d_F(p_i,z) - d_F(p_{i+1},z)| \leq c_1 \max_{\pi^{-1}B(x, d(x,y))}|F|$, 
hence 
$$|d_{F}(x,z)-d_{F}(y,z)|\leq (c_{1}+\max_{\pi^{-1}(B(x,d(x,y)))}|F|)(d(x,y)+1).$$
The second inequality is proved identically.
\end{proof}

The following statement is essentially the same as \cite[Proposition 3.10(4)]{BPP}, but we give its proof for completeness.

\begin{proposition}\label{Gibbscomparisonofshadows}
Let $F$ be a potential which satisfies the (HC) property. Then there exists $c_1 > 0$ such that for all $r>0$, $x,y\in X$ and $\xi \in Sh_{r}(x,y)$ we have
$$|\beta^{F}_\xi(x,y)-d_{F}(x,y)|\leq 2(c_{1}+\max_{\pi^{-1}(B(y,r))}|F|)(r+1).$$
\end{proposition}
\begin{proof}
Suppose $\xi \in Sh_{r}(x,y)$.
Let $p$ be the closest point to $y$ on $[x,\xi)$, so that $d(p,y)\leq r$.
Then $\beta^{F}_\xi(x,p)=d_{F}(x,p)$ and so by the cocycle property of $\beta_{F}(.,.)$ we have:
\begin{align}
|\beta^{F}_\xi(x,y)-d_{F}(x,y)|  & =  |\beta^{F}_\xi(x,y)-\beta^{F}_\xi(x,p)+d_{F}(x,p)-d_{F}(x,y)|    \\
& \leq |\beta_\xi^{F}(p,y)|+|d_{F}(x,p)-d_{F}(x,y)|.
\end{align}
In view of Lemma \ref{(HC)Corollary} each of the two terms is bounded by
 $(c_{1}+\max_{\pi^{-1}(B(y,r))}|F|)(r+1)$, completing the proof.
\end{proof}
The following version of the shadow lemma for $\kappa_F$ is proved in \cite[Lemma 4.2]{BPP}. 
\begin{proposition}\label{gibbshadow}
Let $F : T^1 X \to \mathbb{R}$ be a potential which satisfies the (HC) condition, and let $o \in X$. Then: 
\begin{itemize}
    \item[(a)] 
For large enough $r>0$ we have for any $g \in \Gamma$
$$\kappa_{F}(Sh_{r}(o,go))\asymp e^{d_{F}(o,go)-v_F d(o,go)}$$
where the implied constant depends only on $r$, $o$.
\item[(b)]
There exists a constant $C$ such that for any $n$
$$\sum_{g\in S_n} e^{d_{F}(o,go)-v_F d(o,go)} \leq C.$$
where $S_n := \{ g \in \Gamma: n-1 \leq d(o,g o) \leq n\}$.
\end{itemize}
\end{proposition}

\subsection{Existence of the F-ake drift}

We will now show: 
\begin{theorem}\label{fakedrift}
Let $\Gamma$ be a countable group of isometries of a $CAT(-1)$ metric space $X$, let $\mu$ be a probability measure on $\Gamma$ with finite exponential moment, and let $o \in X$. Then the limit 
$$l_{F,\mu} :=\lim_{n\to \infty}\frac{d_{F}(o,\omega_n o)}{n}$$ exists and is finite for almost every sample path $\omega$ and is independent of the sample path.
\end{theorem}

We will need the following.
\begin{lemma}\label{decayoffakeshadows}
For all large enough $t>0$ one has the estimate
$$P(|\beta^{F}_{\omega_{\infty}}(o,\omega_{n}o)-d_{F}(o,\omega_{n}o)|>t) \leq 1/t^{4}.$$
\end{lemma}
\begin{proof}
Proposition \ref{Gibbscomparisonofshadows} and the subexponential growth of the potential $F$ provides a subexponentially growing function $f:\mathbb{R}^{+}\to \mathbb{R}^{+}$ such that $|\beta^{F}_{\zeta}(o, go)-d_{F}(o, go)|\leq f(D)$ whenever $\zeta \in Sh_{D}(o,go)$. Proposition \ref{closetogeodesics} provides a $b>1$ such that for all large enough $D$ (independent of $n$) and all $n$ we have 
$P(\omega_\infty \notin Sh_{D}(o, \omega_n o)) \leq b^{-D}$. On the other hand, for large enough $D$ we have $f(D) \leq b^{D/4}$. Consequently 
$P(|\beta^{F}_{\omega_{\infty}}(o,\omega_{n}o)-d_{F}(o,\omega_{n}o)|>b^{D/4}) \leq b^{-D}$. Letting $t=b^{D/4}$ completes the proof.

\end{proof}

For each $n$ define $\beta_n(\omega) := \beta^{F}_{\omega_{\infty}}(o,\omega_{n}o)$.

\begin{lemma}\label{Busemannintegrable}
Each $\beta_n(\omega)$ is $P$-integrable.
\end{lemma}

\begin{proof}
Let $\sigma$ be the left shift on the space of increments.
Note, $$\beta_{n+m}(\omega) = \beta_n(\omega) + \beta_m(\sigma^n \omega).$$
Thus, the $\sigma$ invariance of $P$ implies that $\int \beta_{n}dP=n\int \beta_{1}dP$.
and so it suffices to show $\beta_1$ is $P$-integrable.
For this, it suffices to prove the integrability of:
\begin{itemize}
\item[(a)]
$\omega\to |d_{F}(o,\omega_{1}o)|$ and 
\item[(b)]
$\omega\to |\beta^{F}_{\omega_{\infty}}(o,\omega_{1}o)-d_{F}(o,\omega_{1}o)|$.
\end{itemize}
The second is an immediate consequence of Lemma \ref{decayoffakeshadows}.
For the first, note that Definition \ref{HC} implies for any $a>1$
$$|d_{F}(o,go)|\lesssim a^{d(o,go)}\lesssim  a^{K\Vert g\Vert }$$ for a constant $K$,
so that the integrability follows from the exponential moment assumption on $\mu$.

\end{proof}
\begin{proof}[Proof of Theorem \ref{fakedrift}]
For a sample path $\omega = (\omega_n)$ converging to $\omega_{\infty}\in \partial X$ we have $$\beta^{F}_{\omega_{\infty}}(o,\omega_{n+m}o)=\beta^{F}_{\omega_{\infty}}(o,\omega_{n}o)+ \beta^{F}_{\omega_{\infty}}(\omega_n o, \omega_{n+m} o) = $$
so 
we get
$$\beta_{n+m}(\omega) = \beta_n(\omega) + \beta_m(\sigma^n \omega) $$
where $\sigma$ is the left shift on the space of increments.
Thus, Kingman's subadditive ergodic theorem and the ergodicity of $\sigma$ imply that $\beta^{F}_{\omega_{\infty}}(o,\omega_{n})/n$ almost surely converges as $n\to \infty$ to a constant $\ell_{F, \mu}$ independent of $\omega$. 

The result will now follow by a Borel-Cantelli type argument.
Indeed, let $A_{n}(\omega)=|\beta^{F}_{\omega_{\infty}}(o,\omega_{n}o)-d_{F}(o,\omega_{n}o)|$. Then Lemma \ref{decayoffakeshadows} implies $P(A_{n}(\omega)>\sqrt{n}) \leq 1/n^{2}$ for all large enough $n$. Consequently by the summability of $1/n^{2}$, Borel-Cantelli implies that the set of sample paths $\omega$ such that  $A_{n}(\omega)>\sqrt{n}$ for infinitely many $n$ has measure $0$.
Thus, for almost every sample path $\omega$ we have 
$\lim_{n\to \infty}A_{n}(\omega_n)/n=0$, which implies $d_{F}(o,\omega_n o)/n \to l_{F,\mu}$.
\end{proof}

\subsection{The Guivarc'h inequality for Gibbs measures}

\begin{theorem}
Let $\Gamma$ be a countable group of isometries of a $CAT(-1)$ metric space $X$ and let $\mu$ be a probability measure on $\Gamma$ with finite exponential moment. Then we have
$$h_\mu  \leq \ell_\mu v_F - \ell_{F, \mu} .$$
\end{theorem}

\begin{proof}
Fix $\epsilon > 0$ and $n \geq 1$. Define $S_{k} := \{ g \ : \  d(o, go) \in [\epsilon n  (k-1), \epsilon n k) \}$ for any $k \geq 1$.
$$\sum_k \mu^n(S_k) \epsilon n (k-1) \leq \sum_g \mu^n(g) d(o, go) \leq \sum_k \mu^n(S_k) \epsilon k n$$
so by dividing by $n$ and taking the limit as $n \to \infty$ we have 
$$\ell_\mu \leq \liminf_n \sum_k \mu^n(S_k)  \epsilon k \leq  \limsup_n \sum_k \mu^n(S_k) \epsilon k \leq \ell_\mu + \epsilon.$$ 
For any measure $\mu$ (not necessarily a probability) one has by Jensen's inequality
$$\sum_{g \in A} \mu(g) \log \left( \frac{e^{d_F(o, go)}}{\mu(g)}\right) \leq \mu(A) \log \left( \sum_{g \in A} e^{d_F(o, go)} \right)  - \mu(A) \log \mu(A)$$
so, if we denote $H(\mu):= - \sum_g \mu(g) \log \mu(g)$, we obtain
$$H(\mu^n) + \sum_{g} \mu^n(g) d_F(o, go) \leq \sum_k \mu^n(S_k) \log \left( \sum_{g \in S_k} e^{d_F(o, go)} \right) - \sum_k \mu^n(S_k) \log \mu^n(S_k).$$
Now, by definition of critical exponent there exists $C_\epsilon$ such that 
$$ \sum_{d(o, go) \leq R} e^{d_F(o, go)} \leq C_\epsilon e^{(v_F+\epsilon) R} \qquad \textup{for any }R \geq 0$$
so
$$\log \left( \sum_{g \in S_k} e^{d_F(o, go)} \right) \leq (v_F + \epsilon) \epsilon k n  + \log C_\epsilon.$$
Moreover, 
\begin{align*}
    - \sum_k \mu^n(S_k) \log \mu^n(S_k) & = - \sum_{\mu^n(S_k) \leq \frac{1}{e k^{2}}} \mu^n(S_k) \log \mu^n(S_k) - \sum_{\mu^n(S_k) \geq \frac{1}{e k^{2}}} \mu^n(S_k) \log \mu^n(S_k) \leq \\
& \leq \sum_{k = 1}^\infty \frac{2 \log k + 1}{e k^2} + \sum_{k = 1}^\infty \mu^n(S_k) (2 \log k + 1) \leq C
\end{align*}
is bounded independently of $n$. Hence 
\begin{align*}
h_\mu  + \ell_{F, \mu}  & = \lim_{n \to \infty} \frac{H(\mu^n) + E_{\mu^n}[d_F(o, go)] }{n} \leq \\
 & \leq (v_F + \epsilon)  \limsup_{n} \sum_k \mu^n(S_k) \epsilon k \leq
 (v_F + \epsilon)(\ell_\mu + \epsilon)
 \end{align*}
so the claim follows by taking $\epsilon \to 0$. 

\begin{remark}
Let us point out that in the previous proof we used that $\mu$ has finite exponential moment only to make sure that 
\begin{equation} \label{E:EdF}
    \lim_{n \to \infty} \frac{E_{\mu^n}[d_F(o, go)]}{n} = \ell_{F, \mu} 
\end{equation} 
exists. Hence, the Guivarc'h inequality for Gibbs measures holds as long as $\mu$ has finite first moment and \eqref{E:EdF} is true.
Moreover, the assumption that $X$ is $CAT(-1)$ is also not strictly needed, as the essential property is that $\ell_{F, \mu}$ exists.
\end{remark}

\end{proof}

\subsection{Entropy and drift for Gibbs measures}

We will now prove the following analogue of Theorem \ref{maintheorem} for Gibbs states,
which is a generalization of Theorem \ref{intro:maintheoremGibbs} in the Introduction.

\begin{theorem}\label{maintheoremGibbs}
Let $X$ be a proper $CAT(-1)$ metric space, and let $\Gamma$ be a geometrically finite group of isometries of $X$. Let $\mu$ be a probability measure generating $\Gamma$, and let $\nu_\mu$ be the hitting measure of its random walk. Let $F : T^1 X \to \mathbb{R}$ be a potential which satisfies the (HC) property, and let $\kappa_F$
be the corresponding Gibbs density. 
Then 
$$h_\mu \leq \ell_\mu v_F -l_{F,\mu}.$$
Moreover, the following conditions are equivalent.
\begin{enumerate}
    \item The equality
    $$h_\mu =\ell_\mu v_F -l_{F,\mu}$$ holds.
    \item The measures $\nu_{\mu}$ and $\kappa_{F}$ are in the same measure class.
    \item The measures $\nu_{\mu}$ and $\kappa_{F}$ are in the same measure class with Radon-Nikodym derivatives bounded from above and below.
    \item For any basepoint $o \in X$, there exists $C\geq 0$ such that for every $g \in \Gamma$, $$|d_G(e,g)-v_F d(o,go)+d_{F}(o,go)|\leq C.$$

\end{enumerate}
\end{theorem}

Let us start with the proof. For a sample path $\omega$ let $\omega_n$ be its $n$-th position.
Define then
$$\phi_n=\phi_n(\omega)=\frac{\kappa_{F}(Sh_{r}(o, \omega_n o))}{\nu(Sh_{r}(o, \omega_n o))}.$$
Let $\psi_n=\log \phi_n$.
We have the following analogue of Proposition \ref{lemma4.14Haissinsky}.
\begin{proposition}\label{lemma4.14HaissinskyGibbs}
There exists $C_1>0$ such that for any $N\geq 1$ we have 
$$\frac{1}{N}\sum^{N}_{n=1}E(\phi_n)\leq C_1.$$
\end{proposition}
\begin{proof}
The proof is the same as that of Proposition \ref{lemma4.14Haissinsky} except in the end one uses Proposition \ref{gibbshadow} to conclude that 
$$\sum_{g:d(o,go)\leq kN}\kappa_{F}(Sh_{r}(o,go))\lesssim N.$$

\end{proof}
\begin{lemma}\label{HCproductsubexp}
There is a function $f:\mathbb{R}^+\to \mathbb{R}^+$ of subexponential asymptotic growth (i.e. such that $\lim_{r \to \infty} f(r)/c^r = 0$ for all $c>1$)
such that $d(g_{1}o,[o,g_{2}o])<D$ implies $$|d_{F}(o, g_{2}o)-d_{F}(o, g_{1}o)-d_{F}(g_{1}o,g_{2}o)| \leq f(D)$$
for all $g_{1},g_{2}\in \Gamma$.
\end{lemma}
\begin{proof}
Let $p\in [o,g_{2}o]\cap N_{D}(g_{1}o)$.
Then 
$$|d_{F}(o,g_{2}o)-d_{F}(o, g_{1}o)-d_{F}(g_{1}o,g_{2}o)|\leq$$ 
$$ |d_{F}(o, g_{2}o)-d_{F}(o, p)-d_{F}(p, g_{2}o)|+|d_{F}(o,p)-d_{F}(o, g_{1}o)|+|d_{F}(p,g_{2}o)-d_{F}(g_{1}o,g_{2}o)|.$$
The first summand is simply zero since $o,p,g_{2}o$ lie on a geodesic in that order.
Moreover, by Lemma \ref{(HC)Corollary} and the $\Gamma$-invariance of $F$, the second summand is bounded by $(D+1) (c_1 + \max_{B_{D}(o)}|F|)$, and the same is true for the third summand.
Furthermore, the quantity $D \max_{B_{D}(o)}|F|$ is subexponential in $D$ by Definition \ref{HC} b).
\end{proof}

\begin{proposition}\label{lemma4.15HaissinskyGibbs}
There exists $C_2>0$ such that the sequence $E(\psi_n)+C_2$ is sub-additive and $\psi_{n}/n$ converges to $h_\mu  -\ell_\mu v_F +\ell_{F, \mu} $ almost surely and in expectation.  
\end{proposition}

\begin{proof}
By the shadow lemmas  Proposition~\ref{gibbshadow} and Proposition~\ref{harmonicshadow},
$$\psi_n \asymp_{+} d_{G}(e, \omega_n)-v_F d(o,\omega_n o)+d_{F}(o,\omega_n o).$$
Thus, $\psi_{n}/n$ converges to $h_\mu -\ell_\mu v_F +\ell_{F, \mu} $ almost surely and in expectation.
Note that, since $d_G$ satisfies the triangle inequality,
$$\psi_{n+m}-\psi_{n}-\psi_{m}\lesssim_{+}$$
$$\lesssim v_F (d(o, \omega_{n}o)+d(o, \omega_{m}o)-d(o, \omega_{n+m}o))+d_{F}(o, \omega_{n+m}o)-d_{F}(o, \omega_{n}o)-d_{F}(o, \omega_{m}o).$$
Furthermore, $$E(d(o, \omega_{n}o)+d(o, \omega_{m}o)-d(o, \omega_{n+m}o)) = $$
$$ = E(d(o, \omega_{n}o)+d( \omega_{m}o, \omega_{n+m} o)-d(o, \omega_{n+m}o)) = $$
$$ = E(2(o, \omega_{n+m} o)_{\omega_n o})$$ 
so the expectation is uniformly bounded by Proposition \ref{closetogeodesics}. Now, it remains to bound $E(d_{F}(o, \omega_{n+m}o)-d_{F}(o, \omega_{n}o)-d_{F}(o, \omega_{m}o)).$ 

Lemma \ref{HCproductsubexp} provides a subexponential function $f$ such that $d(\omega_{n}o,[o,\omega_{n+m}o])<D$ implies $|d_{F}(o, \omega_{n+m}o)-d_{F}(o, \omega_{n}o)-d_{F}(\omega_{n}o, \omega_{n+m}o)|<f(D)$ for all $g_{1},g_{2}\in \Gamma$.
Thus, Proposition \ref{closetogeodesics} implies that there is a constant $C$ with 
$$P(|d_{F}(o, \omega_{n+m}o)-d_{F}(o, \omega_{n}o)-d_{F}(\omega_n o, \omega_{n+m}o)|>f(D))<Ce^{-D/C}$$ 
for all $n,m\geq 1$ and $D>0$. Since $f$ has subexponential growth, this implies $E(d_{F}(o, \omega_{n+m}o)-d_{F}(o, \omega_{n}o)-d_{F}(o, \omega_{m}o))$ is bounded above independently of $m,n$, completing the proof.
\end{proof}

\begin{proposition}\label{lemma4.16HaissinskyGibbs}
\begin{enumerate}[a)]
    \item If $\kappa_F$ and $\nu_{\mu}$ are not equivalent, then $\phi_n$ tends to $0$ in probability.
    \item If $\kappa_F$ and $\nu_{\mu}$ are equivalent then $\frac{\log \kappa(Sh_{r}(o,\omega_n o))}{\log \nu(Sh_{r}(o,\omega_n o))}$ tends to 1 in probability.
\end{enumerate}
\end{proposition}

\begin{proof}
Identical to the proof of Proposition \ref{lemma4.16Haissinsky}.
\end{proof}

\begin{theorem}\label{(1)implies(2)Gibbs}
The measures $\kappa_F$ and $\nu_{\mu}$ are equivalent if and only if 
$$h_\mu =\ell_\mu v_F - l_{F,\mu}.$$ 
\end{theorem}

\begin{proof}
Same as proof of Theorem \ref{(1)implies(2)}.
\end{proof}

Recall that $\widecheck{\mu}$ is the reflection of $\mu$, and we denote as $\widecheck{F}=F\circ \iota$ the reflected potential.
\begin{corollary}\label{reflectedequivalenceGibbs}
The measure $\nu_{\widecheck{\mu}}$ is equivalent to $\kappa_{\widecheck{F}}$ if and only if $\nu_{\mu}$ is equivalent to $\kappa_{F}$.
\end{corollary}
\begin{proof}
Note that $h_{\widecheck{\mu}} = h_\mu $, $l_{\widecheck{\mu}} = \ell_\mu $ and $v_{\widecheck{F}}=v_F $. Furthermore, $l_{\widecheck{F},\widecheck{\mu}}=l_{F,\mu}$.
Consequently $h_\mu =\ell_\mu v_F - l_{F,\mu}$ if and only if $h_{\widecheck{\mu}} = l_{\widecheck{\mu}}v_{\widecheck{F}}- l_{\widecheck{F},\widecheck{\mu}}$.
Theorem \ref{(1)implies(2)Gibbs} now implies the result.
\end{proof}

\begin{proposition}\label{(2)implies(3)Gibbs}

If  $\nu_{\mu}$ is equivalent to $\kappa_F$, then the Radon-Nikodym derivative $d\kappa_{F}/d\nu_{\mu}$ is bounded away from 0 and infinity.
\end{proposition}
\begin{proof}
If $\nu_{\mu}$ is equivalent to $\kappa_{F}$ then $\nu_{\check{\mu}}$ is equivalent to $\kappa_{\tilde{F}}$.
We need to show that $\kappa_{F} \otimes \kappa_{\widecheck{F}}$ can be scaled by a bounded function $f$ to give a $\Gamma$-invariant Radon measure on $\partial^{2}X$. This is done in \cite[Equation 4.4]{BPP}. The proof is now the same as that of Proposition \ref{(2)implies(3)}.
\end{proof}

\begin{proposition}\label{(3)implies(4)Gibbs}
If $\kappa_F$ and $\nu_{\mu}$ are equivalent, then $$|d_{G}(g,g')-v_F d(go,g'o)+d_{F}(go,go')|$$ 
is uniformly bounded independently of $g,g'\in \Gamma$.
\end{proposition}
\begin{proof}
If $\kappa_F$ and $\nu_{\mu}$ are equivalent then their Radon-Nikodym derivative is bounded away from $0$ and infinity. Consequently, the ratio satisfies 
$$C^{-1}\leq \frac{\kappa_F(Sh_r(o, go))}{\nu_\mu(Sh_r(o, go))} \leq C$$ 
for some $C>0$ independent of $g$.
The shadow lemmas Propositions \ref{harmonicshadow} and \ref{gibbshadow} now imply the result, together with the fact that all metrics we use are $\Gamma$-invariant.
\end{proof}

\subsection{Growth of parabolics and singularity of harmonic measure}

We will prove the following.
\begin{proposition}\label{exponentialdistortionparabolic}
Let $P<\Gamma$ be a parabolic subgroup.  There are $c >1$, $D > 0$ such that 
$\Vert g\Vert  \geq D c^{d(o,g o)}$ for all $g \in P$.
\end{proposition}
Together with Theorem \ref{maintheoremGibbs} this will imply:
\begin{corollary}
If $\Gamma \curvearrowright X$ has parabolics then $\kappa_{F}$ and $\nu_\mu$ are mutually singular.
\end{corollary}
It remains to prove Theorem \ref{exponentialdistortionparabolic}.
First, recall that Osin \cite[Proposition 2.27]{Osin} showed that if $\Gamma$ is finitely generated so are the stabilizers of any parabolic point of $\partial X$, called maximal parabolic subgroups.
Choose a symmetric finite generating set for $P$ and let $\Vert .\Vert _P$ be the associated word metric on $P$.
The following can be found in Drutu-Sapir \cite{DS} or Gerasimov-Potyagailo \cite[Corollary 3.9]{GP}.

\begin{lemma}\label{parabolicquasiconvex}
A maximal parabolic subgroup $P<\Gamma$ is quasi-convex.
In particular, $\Vert g\Vert _{P}\asymp \Vert g\Vert _{\Gamma}$ for all $g\in P$.
\end{lemma}
The following can be found in Bridson-Haefliger \cite[Proposition I.8.25]{Bridson-Haefliger}
in the context of CAT(0) spaces.
\begin{lemma}\label{parabolicpreservehorosphere}
Let $P$ be a maximal parabolic subgroup
which stabilizes $\zeta \in \partial X$. Then for any $g\in P$ and $x,y\in X$ we have $\beta_{\zeta}(x,gy)=\beta_{\zeta}(x,y)$.
\end{lemma}

For $x\in X$ and $\zeta \in \partial X$ the horosphere $\mathcal{H}_{\zeta}(x)$ through $x$ centered at $\zeta$ is defined to be  the set of $z\in X$ with $\beta_{\zeta}(x,z)=0$. 
The associated (open) horoball $\mathcal{B}(x,\zeta,t)$ is defined to be  the set of $z\in X$ with $\beta_{\zeta}(z,x)< -t$.

\begin{proposition} \label{horoballcontainsball}
There exists a $K>0$ depending only on the hyperbolicity constant of $X$ such that if $\beta_{\zeta}(x,z)=0$ and $d(x, z) \geq 2K + 4r$ then $\mathcal{B}(x, \zeta,r)$ contains the ball of radius
 $d(x,z)/2-K-2r$ around the midpoint of $[x,z]$.
\end{proposition}

\begin{proof}
Let $\alpha$ be a geodesic in $X$ from $x$ to $\zeta$. By definition, $H_{\zeta}(x)$ is the Gromov-Hausdorff limit of spheres $S(\alpha(n),n)$ of radius $n$ centered at $\alpha(n)$.
Furthermore, $\mathcal{B}(x, \zeta,r)$ is the limit of balls $B(\alpha(n+r),n)$ 
Thus, it suffices to show that for large enough $t$, for any $z\in S(\alpha(t),t)$, $B(\alpha(t+r),t)$ contains the ball of radius $d(x,z)/2-K-r$ around the midpoint of $[x,z]$. Consider the geodesic triangle with vertices $x,z,\alpha(t)$.
By Gromov hyperbolicity there is a $p\in X$ which is within the hyperbolicity constant $\delta$ of all three sides of the triangle.
Then $t=d(x,\alpha(t))\asymp_{+,\delta} d(x,p)+d(p,\alpha(t))$ and 
$t=d(z,\alpha(t))\asymp_{+,\delta} d(z,p)+d(p,\alpha(t))$.
Thus, $d(x,p)\asymp_{+,\delta} d(z,p)$ so $p$ is within $3\delta$ of a midpoint $q$ of $[x,z]$. Thus, $$d(q,\alpha(t))\asymp_{+,5\delta} t-d(x,q) = t- d(x,z)/2.$$
It follows that $d(q,\alpha(t+r))\leq t+r+5\delta-d(x,z)/2$.
The result follows with $K=6\delta$.

\end{proof}
We are now ready to prove Proposition  \ref{exponentialdistortionparabolic}.
Indeed, by hyperbolicity of $X$ (see e.g. \cite[Proposition III.H.1.6]{Bridson-Haefliger}) there are $c>1,D>1$ such that for any $x,z,y$ along an $X$ geodesic (in that order)  any path from $x$ to $y$ disjoint from $B_{R}(z)$ has length at least $D c^{R}$. Consider a maximal parabolic subgroup $P$ and $g\in P$. 

Let $S_P$ be a finite generating set for $P$, and $T :=\max_{p\in S_P}d(o,po)$.
Let $p_1,...,p_n$ be an $(P,S_P)$ geodesic from $p_0=e$ to  $p_n=g$. For each $i$, let $\gamma_i$ be a geodesic in $X$ from $p_{i}o$ to $p_{i+1}o$, and let $\gamma$ be the concatenation of the $\gamma_i$. 

Then $\gamma$ is a path in $X$ from $o$ to $go$ outside of $\mathcal{B}(o, \zeta,T)$ of length at most $T\Vert g\Vert _P\asymp \Vert g\Vert $. In particular, this path does not intersect the ball of radius $d(o,go)/2-K-2T$ about the midpoint of $[o,go]$.
It follows that $\Vert g\Vert $ is bounded below by a constant times $Dc^{d(o,go)/2-K-2T}$ completing the proof.

\section{Appendix}

\subsection{Exponential deviation estimates}

In this section we prove Proposition \ref{closetogeodesics}.

We assume $\Gamma \curvearrowright X$ is a non-elementary action on a proper geodesic Gromov hyperbolic space.
Furthermore, $\mu$ is a probability measure on $\Gamma$ with exponential moment and support generating $\Gamma$ as a semigroup.
\begin{proposition}\ref{closetogeodesics}.
Let $\nu$ be the $\mu$-stationary measure on $\partial X$. 
For each $o\in X$ there exists a $C>0$ such that for each $0 \leq k \leq n$ and $a>1$ we have
$$P(d(\omega_k o,[o,\omega_n o])>a) \leq Ce^{-a/C}$$
and
$$P(d(\omega_k o,[o,\omega_\infty))>a) \leq Ce^{-a/C}$$
where $[o,\omega_n o]$ and $[o,\omega_\infty)$ are any geodesics connecting the respective endpoints.
\end{proposition}

To prove the proposition we will need the following lemmas. 

\begin{lemma} \label{BQlemma}\cite[Remark 4.4]{BQhyp}
There is a $t>0$ such that 
$$\sup_{y\in \partial X} \int_{x\in \partial X} e^{t\rho_{o}(x,y)}<\infty.$$
\end{lemma}

\begin{lemma}\label{boundarytointerior}
The same is true if the supremum is taken over all $y\in 
\Gamma o \cup \partial X$. 
\end{lemma}

\begin{proof}
Note for $o,o'\in X$ and $x \in X\cup \partial X$ and $g\in \Gamma$ we have $|\rho_{o'}(x,go')-\rho_{o}(x,go)|\lesssim_{+} 2d(o,o')$.
It thus suffices to prove the lemma for any particular basepoint $o\in X$. Since $X$ is proper Gromov hyperbolic it has a bi-infinite geodesic $\alpha$. We assume without loss of generality that $o\in \alpha$. This means any $go$ lies on a bi-infinite geodesic $g\alpha$.

We claim that for any $y\in X$, $x\in \partial X$ and any bi-infinite geodesic $\alpha$ containing $y$ we have $$\rho_{o}(x,y)\lesssim_{+}2 \max(\rho_{o}(x,\alpha_{+}),\rho_{o}(x,\alpha_{-})).$$
To that end let $p_+$ and $p_-$ be points on $(x,\alpha_+)$ and $(x,\alpha_-)$ respectively at minimal distance from $o$. By Gromov hyperbolicity, each $p_{\pm}$ is within the hyperbolicity constant $\delta$ of either $[y,x]$ or $[y,\alpha_{\pm})\subset (\alpha_{-},\alpha_{+})$. If $p_{+}$ is within  $\delta$ of  $[y,x)$ then $d(o,[y,x))\leq d(o,p_{+})+\delta=d(o,(x,\alpha_{+})) +\delta$ and so $\rho_{o}(x,y)\lesssim_{+}\rho_{o}(x,\alpha_{+})$.
Similarly if $p_{-}$ is within $\delta$ of  $[y,x)$ then
 $\rho_{o}(x,y)\lesssim_{+}\rho_{o}(x,\alpha_{-})$.
We are left to consider the case where each $p_{\pm}$ is within $\delta$ of some $q_{\pm}\in [y,\alpha_{\pm})$.
Let $D= \max (d(p_{+},o),d(p_{-},o))$.  Then $d(p_{+},p_{-}) \leq 2D$ and so $d(q_{+},q_{-}) \leq 2D + 2\delta$. Hence, at least one of $q_{\pm}$, say $q_+$ is within $D+\delta$ of $y$. Thus $p_{+}$ is within $D+2\delta$ of $y$.
Consequently, $d(o,[y,x))\leq d(o,p_+)+d(p_{+},y)\leq 2D+2\delta$.
Hence $\rho_{o}(x,y)\lesssim D$, proving the claim.

Now, by the claim we have 
 $$\int_{x\in \partial X} e^{t\rho_{o}(x,go)}d\nu(x)\lesssim \max \left(\int_{x\in \partial X} e^{2t\rho_{o}(x,g\alpha_{+})}d\nu(x), \int_{x\in \partial X} e^{2t\rho_{o}(x,g\alpha_{-})}d\nu(x) \right)$$
so we conclude using Lemma  \ref{BQlemma}.
\end{proof}

To simplify notation, we will from now on denote $|g| := d(go, o)$. The following lemma is due to Sunderland.

\begin{lemma}\label{Sunderland}\cite[Criterion 11]{Sunderland}
There is an $N_0>0$ such that for all $N\geq N_0$ there is a $t(N)>0$ and an $\epsilon>0$ such that for $0<t<t(N)$ we have 
$$\sup_{g\in \Gamma} E(e^{-t(|g\omega_N|-|g|)})<1-\epsilon.$$
\end{lemma}

This implies:
\begin{lemma}\label{expdecayexpectation}
There is a $t_0>0$ and $C>0$ such that 
$$E(e^{-t(|\omega_{N+k}|-|\omega_{k}|)})<Ce^{-N/C}$$
for all $N,k\geq 0$ and $0<t<t_0$.
\end{lemma}

\begin{proof}
Let $N_0$ be given by Lemma \ref{Sunderland}. It suffices to prove that for some $t$ 
\begin{enumerate}
    \item 
$\sup_{k\geq 0}E(e^{-t(|\omega_{N+k}|-|\omega_{k}|)})<\infty$ for $N\leq N_0$
    \item 
$E(e^{-t(|\omega_{N+N_0+k}|-|\omega_{k}|)})\leq (1-\epsilon)E(e^{-t(|\omega_{N+k}|-|\omega_{k}|)})$ for $N,k\geq 0$
\end{enumerate}
since then the claim follows by induction.

For the first claim, note that $|\omega_{N+k}|-|\omega_{k}|\geq -|\omega^{-1}_{k}\omega_{N+k}|$, which has the same distribution as $-|\omega_N|$.
Thus, 
$$E(e^{-t(|\omega_{N+k}|-|\omega_{k}|)}) \leq  E(e^{t|\omega_N|})\leq E(e^{t|\omega_1|})^{N}\leq E(e^{t|\omega_1|})^{N_0}$$ 
which for small enough $t$ is finite by the exponential moment assumption.

We now prove the second claim.
Indeed, $$E(e^{-t(|\omega_{N+N_0+k}|-|\omega_{k}|)})=E(e^{-t(|\omega_{N+k}|-|\omega_{k}|)} e^{-t(|\omega_{N+N_0+k}|-|\omega_{N+k}|)}).$$
Conditioning on $\omega_{N+k}=h,\omega_k=g$ the last expression becomes 
$$\sum_{g,h\in \Gamma} e^{-t(|h|-|g|)}E(e^{-t(|h\omega_{N_0}|-|h|)}) P(\omega_{N+k}=h,\omega_k=g).$$ By Lemma \ref{Sunderland}, we have $E(e^{-t(|h\omega_{N_0}|-|h|)})<(1-\epsilon)$ for any $t < t(N_0)$ and thus
$$E(e^{-t(|\omega_{N+N_0+k}|-|\omega_{k}|)}) <  (1-\epsilon)E(e^{-t(|\omega_{N+k}|-|\omega_{k}|)}).$$
\end{proof}

\begin{lemma} \label{linexp}
There is a $C>0$ such that for all $0 \leq k \leq n$ we have 
$$P(d(\omega_n o,o)-d(\omega_k o,o)<(n-k)/C)<Ce^{-(n-k)/C}.$$
\end{lemma}

\begin{proof}
The exponential Markov inequality implies for any $D>0$ and any $N, k \geq 0$: 
$$P(|\omega_{N+k}|-|\omega_{k}|<DN)=P(e^{-t(|\omega_{N+k}|-|\omega_{k}|)}>e^{-tDN})$$ 
hence by Lemma \ref{expdecayexpectation} $$<e^{tDN}E(e^{-t(|\omega_{N+k}|-|\omega_{k}|)})
<Ce^{tDN-N/C}<2Ce^{-N/(2C)}$$ for small enough $D$, 
so the claim follows by setting $N = n-k$.
\end{proof}

\begin{lemma}
There is a $C>0$ such that for all $0 \leq k \leq n$ and $a>0$
we have $P(d(\omega_{k}o,o)-d(\omega_{n}o,o)>a)<Ce^{-a/C}$.
\end{lemma}

\begin{proof}

Let $t>0$ be smaller than the $t_0$ in Proposition \ref{linexp} and also small enough so that $E(e^{t|\omega_1|})<\infty$.
Let $D=E(e^{t|\omega_1|})$.

Suppose $n-k \leq ta/(2 \log D)$. 
Then we have $E(e^{t(|\omega_k|-|\omega_n|)})\leq E(e^{t|\omega^{-1}_{k}\omega_{n}|})=E(e^{t|\omega_{n-k}|})\leq E(e^{t|\omega_1|})^{n-k}\leq D^{n-k}$.
Consequently by the Markov inequality 
$P(|\omega_k|-|\omega_n|>a) = P(e^{t(|\omega_k| - |\omega_n|)} > e^{ta}) \leq e^{-ta} E(e^{t(|\omega_k|-|\omega_n|)}) \leq D^{n-k}/e^{ta} \leq e^{-ta/2}$.

On the other hand if  $n-k \geq Ka= ta/(2 \log D)$ then Proposition \ref{linexp} implies that 
$P(d(\omega_{n}o,o)- d(\omega_{k}o,o) < - a)\leq  P(d(\omega_{n}o,o) - d(\omega_{k}o,o) < (n-k)/C) \leq Ce^{-(n-k)/C} \leq Ce^{-Ka/C}$.
\end{proof}

\begin{proof}[Proof of Proposition \ref{closetogeodesics}]

We first prove the second statement of Proposition \ref{closetogeodesics}.
By Lemma \ref{boundarytointerior}, we obtain a $C > 0$ such that 
$$\nu(\zeta\in \partial X: \rho_{o}(go,\zeta)>a) \leq Ce^{-a/C}$$
for any $g \in \Gamma$ and any $a > 0$.
Therefore, for each $n$ we get:
$$P(\rho_{\omega_{n}o}(o,\omega_{\infty})>a))=\sum_{g\in \Gamma}P(\rho_{\omega_{n}o}(o,\omega_{\infty})>a)|\omega_n=g)P(\omega_n=g)=$$
$$=\sum_{g\in \Gamma}P(\rho_{o}(g^{-1}o,\omega_{\infty})>a)P(\omega_n=g)\leq Ce^{-a/C}.$$
We have thus proved that
$P(d(\omega_n o,[o,\omega_\infty))>a)<Ce^{-a/C}.$

We now prove the first statement of Proposition \ref{closetogeodesics}.
Consider $0 \leq k \leq n$. The events 
$d(\omega_n o,[o,\omega_\infty))<a$, $d(\omega_k o,[o,\omega_\infty))<a$ and $d(\omega_{n}o,o)-d(\omega_{k}o,o)>-a$ all have probability at least $1-Ce^{-a/C}$ where $C$ does not depend on $k,n$. 
Suppose these events hold.
Let $p_{k},p_{n}\in [o,\omega_{\infty})$ be at minimal distance from $\omega_{k}o$ and $\omega_{n}o$ respectively.
Then by the triangle inequality  $d(o,p_{n})-d(o,p_{k})>-3a$.

If $d(o,p_{n})\geq d(o,p_{k})$ then by the fellow traveling property $d(p_{k},[o,\omega_{n}o])<Kd(\omega_{n}o,p_{n})<Ka$ for a constant $K$ which only depends on the hyperbolicity constant of $X$.

On the other hand, if $d(o,p_{n})\leq d(o,p_{k})\leq d(o,p_{n})+3a$ then $d(p_{k},[o,\omega_{n}o])\leq 3a+d(\omega_{n}o,p_{n})\leq 4a$.

In either case, we have $d(\omega_{k}o,[o,\omega_{n}o]) \leq (4+K)a$.
Thus we obtain $$P(d(\omega_{k}o,[o,\omega_{n}o])>(4+K)a)<3Ce^{-a/c}$$ completing the proof of Proposition \ref{closetogeodesics}.

\end{proof}

\subsection{The Green metric for non-symmetric measures}
We will prove the following.

\begin{proposition}\label{Greendecay}
Let $\mu$ be a generating probability measure with finite exponential moment on a nonamenable group $\Gamma$. Let $G$ be the associated Green's function. Then there is a constant $C>0$ such that 
$$e^{-C \Vert x^{-1}y \Vert}/C\leq G(x,y)\leq Ce^{-\Vert x^{-1}y \Vert/C}$$
for any $x, y \in \Gamma$. In other words, the Green metric $d_G(x,y)=-\log \frac{G(x,y)}{G(e,e)}$ is quasi-isometric to the word metric.
\end{proposition}
The lower bound is immediate from the Harnack inequality. We thus just need to consider the upper bound.
This was proved for symmetric measures on non-elementary hyperbolic groups in \cite[Proposition 3.6]{BlachereHassinskyMathieu2} and \cite[Lemma 4.2]{Haissinsky}. The proof there carries over without modification for symmetric measures on nonamenable groups. However, for non-symmetric measures an extra argument is required.

\begin{lemma}\cite[Proposition IV.4]{Coulhon} \label{Coulhon}
Let $\mu$ and $\mu_0$ be two generating probability measures on a group $\Gamma$. Assume that  $\mu_0$ is symmetric and that there exists $C>0$ such that  $\mu_0 (g)\leq C \mu(g)$ and  $\mu^{n}_{0}(g)\leq Ce^{-n/C}$ for all $g\in \Gamma$. Then there is  a $C'>0$ such that
 $\mu^{n}(g)\leq C'e^{-n/C'}$ for all $g\in \Gamma$. 
\end{lemma}

\begin{lemma}\label{transitiondecay}
Let $\mu$ be a generating measure on a nonamenable group $\Gamma$.
There is a $C>0$ such that for all $n\geq 0$ and $g\in \Gamma$ we have 
$\mu^{n}(g)\leq Ce^{-Cn}$. 
\end{lemma}

\begin{proof}
We first prove the lemma for symmetric measures. 
This is in fact done in \cite[Lemma 3.6]{BlachereHassinskyMathieu2} (where it is stated for hyperbolic groups, but the proof only uses amenability). We repeat it for completeness.
Since $\Gamma$ is nonamenable, the spectral radius of the random walk is less than 1, so we have $\mu^{n}(e)\leq Ce^{-Cn}$ for all $n$.
Indeed, in the symmetric case by the Cauchy-Schwarz inequality
$$\mu^{2n}(x)=\sum_{y\in \Gamma}\mu^{n}(y) \mu^{n}(y^{-1}x)\leq \sqrt{\sum_{y\in \Gamma} \mu^{n}(y)^2 } \sqrt{\sum_{y\in \Gamma} \mu^{n}(y)^{2}}= \sum_{y\in \Gamma} \mu^{n}(y)^{2}.$$ 
By the symmetry of $\mu$ we have
$$\sum_{y\in \Gamma} \mu^{n}(y)^{2}=\sum_{y\in \Gamma} \mu^{n}(y)\mu^{n}(y^{-1})=\mu^{2n}(e).$$
Thus, $\mu^{2n}(x)\leq \mu^{2n}(e)$. Similarly,  
$$\mu^{2n+1}(x)=\sum_{y\in \Gamma}\mu(y)\mu^{2n}(y^{-1}x)\leq \sum_{y\in \Gamma}\mu(y)\mu^{2n}(e)=\mu^{2n}(e).$$
Thus, we have proved the lemma for symmetric measures. 

Now, suppose $\mu$ is not symmetric. Let $\mu_0$ be any symmetric generating measure on $\Gamma$ with finite support. Let $K>0$ be an odd number such that $\bigcup^{K}_{i=1} \textup{supp}(\mu^{i})$ contains $\textup{supp}(\mu_0)$. Let $\widetilde{\mu} :=\frac{\sum^{K}_{i=1}\mu^{i}}{K}.$ 
Then we have  $\widetilde{\mu}(g)\geq c\mu_{0}(g)$ for every $g$, for $c=\min \left\{\frac{\widetilde{\mu}(g)}{\mu_0(g)}: g\in \textup{supp}(\mu_{0})\right\}>0$.
Thus Lemma \ref{Coulhon} and the bound on $\mu^{n}_{0}(g)$ implies that $\widetilde{\mu}^n(g)\leq C'e^{-n/C'}$ for a constant $C'$.  

Let $M(K,n)$ be the coefficient of $x^{n(K+1)/2}$ in the polynomial $(x+ x^2 + \dots +x^K)^{n}$.  Then for all $g\in \Gamma$ we have 
$\widetilde{\mu}^{n}(g)\geq \frac{M(K,n)}{K^{n}} \mu^{n(K+1)/2}(g)$.
Note that $\frac{M(K, n)}{K^n}$ equals the probability $P(X_1 + \dots + X_n = \frac{n(K+1)}{2})$
where $(X_i)$ are i.i.d. random variables
uniform on $\{1, 2, \dots, K\}$.
Moreover, $X_1$ has mean $\frac{K+1}{2}$.
Now, by the central limit theorem 
$$\liminf_{n \to \infty} \sqrt{n} P(X_1 + \dots + X_K = n (K+1)/2) > 0$$
hence there exists a constant $C_1$ such that
$$\frac{K^n}{M(K,n)} \leq C_1 \sqrt{n}$$
for any $n$. Thus, 
$$\mu^{n(K+1)/2}(g)\leq \frac{K^n}{M(K, n)} \widetilde{\mu}^n(g) \leq C_1 C \sqrt{n} e^{-n/C} \leq  C''e^{-n/C''}$$ 
for a constant $C''$.
Furthermore, for $i\leq (K+1)/2$ we have 
$$\mu^{i + (K+1)n/2}(g)=\sum_{h\in \Gamma}\mu^{(K+1)n/2}(h^{-1}g)\mu^{i}(h)\leq C''e^{-n/C''}.$$ 
This completes the proof.
\end{proof}

The proof of exponential decay of Green's function now follows as in \cite[Proposition 3.6]{BlachereHassinskyMathieu2} or \cite[Lemma 4.2]{Haissinsky}. We reproduce it for completeness.

We assumed that $E_{\mu}[\exp(r\Vert g \Vert)] = E<\infty $ for a given $r>0$.
Note, for any $k\leq n$ 
$$\Vert \omega_k\Vert \leq \sum^{n-1}_{i=1}\Vert \omega_{i}^{-1}\omega_{i+1}\Vert$$
and the increments $\omega_{i}^{-1}\omega_{i+1}$  are independent random variables and all follow the same law
as $\omega_1$. Therefore for
any $b>0$ the exponential Chebyshev inequality implies
$$P \left(\sup_{1\leq k \leq n}\Vert \omega_k\Vert \geq nb \right)\leq e^{-rnb}E\left[ \exp \left(r \sup_{1\leq k \leq n}\Vert \omega_k\Vert \right)\right] \leq e^{-rnb}E^{n}=e^{(-rb+\log E)n}.$$
We choose $b$ large enough so that $c= rb- \log E>0$. Then $$G(e,g)=\sum^{\infty}_{k=1}\mu^{k}(g)=\sum_{1\leq k\leq \Vert g\Vert /b}\mu^{k}(g)+\sum_{k >  \Vert g\Vert /b}\mu^{k}(g).$$
By Lemma \ref{transitiondecay} we have for the second summand the bound:
$$\sum_{k\geq \Vert g\Vert /b}\mu^{k}(g)\leq \sum_{k\geq \Vert g\Vert /b}Ce^{-k/C}\leq De^{-\Vert g\Vert /D}$$ for some constant $D>0$.
Meanwhile, for the first summand we have the bound
$$\sum_{1\leq k\leq \Vert g\Vert /b}\mu^{k}(g)\leq \frac{\Vert g\Vert }{b}\sup_{1\leq k\leq \Vert g\Vert /b}\mu^{k}(g)
\leq \frac{\Vert g\Vert }{b} P \left(\sup_{1\leq k\leq \Vert g\Vert /b}\Vert \omega_{k}\Vert \geq \Vert g\Vert \right)\leq \Vert g \Vert e^{-c\Vert g\Vert /b}.$$
As both summands decay exponentially in $\Vert g\Vert $ the proof is complete.

\bibliographystyle{plain}
\bibliography{gibbs}

\end{document}